\documentclass[a4paper]{article}
\usepackage[colorlinks=true,citecolor=black,linkcolor=black]{hyperref}
\usepackage{geometry}
\usepackage{amsmath,amssymb,amsthm}
\usepackage{color}
\usepackage{hyperref, doi}
\usepackage{subfig} %
\usepackage{graphicx}
\usepackage{booktabs} 
\usepackage{array} 
\usepackage{paralist} 
\usepackage{verbatim} 
\usepackage[font=small]{caption}

\newcommand{\mP}{\mathcal{P}}

\DeclareMathOperator{\Div}{div}

\DeclareMathOperator{\SRBGS}{SRBGS}
\DeclareMathOperator{\BD}{BD}
\newcommand{\symgrad}{\mathcal{E}}

\newcommand{\grad}{\nabla}

\numberwithin{equation}{section}

\theoremstyle{theorem}
\newtheorem{lemma}{Lemma}[section]
\newtheorem{theorem}{Theorem}[section]
\newtheorem{proposition}{Proposition}[section]
\newtheorem{corollary}{Corollary}[section]
\theoremstyle{remark}
\newtheorem{remark}{Remark}[section]
\theoremstyle{definition}

\DeclareMathOperator{\dom}{dom}

\DeclareMathOperator*{\wlim}{w-lim}
\DeclareMathOperator{\erg}{erg}

\DeclareMathOperator*{\argmin}{arg\,min}
\newcommand{\seq}[1]{\{{#1}\}}

\newcommand{\norm}[2][]{\|{#2}\|_{#1}}
\newcommand{\scp}[3][]{\langle{#2},{#3}\rangle_{#1}}

\newcommand{\subgrad}{\partial}
\newcommand{\RR}{\mathbf{R}}

\newcommand{\mM}{\mathcal{M}}
\newcommand{\mA}{\mathcal{A}}
\newcommand{\mL}{\mathcal{L}}
\newcommand{\mO}{\mathcal{O}}

\newcommand{\gap}{\mathfrak{G}}

\DeclareMathAlphabet{\mathbfit}{OML}{cmm}{b}{it}
\DeclareMathOperator{\TV}{TV}
\DeclareMathOperator{\TGV}{TGV}

\begin{document}
\title{Analysis of Fully Preconditioned ADMM with Relaxation in Hilbert Spaces with Applications to Imaging}
\author{Hongpeng Sun\thanks{Institute for Mathematical Sciences,
Renmin University of China, No.~59, Zhongguancun Street, Haidian District,
100872 Beijing, People's Republic of China.
Email: \href{mailto:hpsun@amss.ac.cn}{hpsun@amss.ac.cn}.}}
\date{}

\maketitle

\begin{abstract}
Alternating direction method of multipliers (ADMM) is a powerful first order methods for various applications in inverse problems and imaging. However, there is no clear result on the weak convergence of ADMM in infinite dimensional Hilbert spaces with relaxation studied by Eckstein and Bertsakas \cite{EP}. In this paper, by employing a kind of ``partial" gap analysis, we prove the weak convergence of general preconditioned and relaxed ADMM in infinite dimensional Hilbert spaces, with preconditioning for solving all the involved implicit equations under mild conditions. We also give the corresponding ergodic convergence rates respecting to the ``partial" gap function. Furthermore, the connections between certain preconditioned and relaxed ADMM and the corresponding Douglas-Rachford splitting methods are also discussed, following the idea of Gabay in  \cite{DGBA}.
Numerical tests also show the efficiency of the proposed overrelaxation variants of preconditioned ADMM.
\end{abstract}


\section{Introduction}
%

\date{} 
ADMM (alternating direction method of multipliers) is one of the most popular first order methods for non-smooth optimization and regularization problems widely used in inverse problems \cite{Boyd, CP3}.  ADMM was first introduced in \cite{Gabay, FM}, and extensive studies can be found in~\cite{Boyd, CP3, DDY, DY, EP, FBX, GQY, HY, LST, MS, XMS} and so on. It aims at the following general constrained optimization problem,
 \begin{equation}\label{saddle:primal:original}
 \min_{u \in X, p \in Y} F(u) + G(p), \quad \text{subject to} \quad Au +B p = c,
 \end{equation}
where $F \in \Gamma_0(X)$, $G \in \Gamma_0(Y)$, $c \in Z$ and $X, Y, Z$ are real Hilbert spaces with $\Gamma_0(\cdot)$ denoting the corresponding proper, convex and lower semi-continuous function spaces~\cite{HBPL}.  In the sequel, we assume $A \in L(X,Z)$, $B \in L(Y,Z)$, where $L(X_1,X_2)$ denotes the spaces of linear and continuous operators mapping from Hilbert spaces $X_1$ to $X_2$. With appropriate conditions \cite{BS1,KK}, this problem is equivalent to the following saddle-point problem, and we assume there exists a bounded saddle-point,
 \begin{equation}\label{eq:saddle-point-prob}
 \min_{u \in X, p \in Y} \max_{\bar \lambda \in Z} \mL(u,p, \bar \lambda), \quad \mL(u,p, \bar \lambda):= F(u) + G(p) +  \langle \bar \lambda, Au+Bp-c \rangle,
 \end{equation}
 where $\mL(u,p, \bar \lambda)$ is the Lagrangian function and $\bar \lambda$ is the Lagrangian multiplier. The
associated Fenchel-Rockafellar dual problem of \eqref{saddle:primal:original} reads as follows \cite{DY},
 \begin{equation}\label{eq:saddle-point-prob:dual}
 \quad \max_{ \bar \lambda\in Z} -F^*(-A^* \bar \lambda)- G^*(-B^* \bar \lambda) - \langle c, \bar \lambda \rangle.
 \end{equation}
Once writing the following augmented Lagrangian,
 \begin{equation}\label{eq:augmented:lag}
  L_{r}(u, p, \bar \lambda,c): = F(u) +  G(p) + \langle \bar{\lambda},  Au+Bp-c\rangle + \frac{r}{2}\| Au + Bp-c \|^2,
\end{equation}
 one gets the ADMM method by solving the sub minimization problems $u$ and $p$ consecutively in \eqref{eq:augmented:lag}, with updating the Lagrangian multipliers at last,
\begin{equation}
    \label{admm:original:sys:final}\tag{ADMM}
    \left\{
    \begin{aligned}
       u^{k+1}
      &= (rA^*A + \partial F)^{-1}[A^*(-rBp^k +rc - \bar{\lambda}^k) ],\\
p^{k+1}  &= (r B^*B + \partial G)^{-1}[B^*( -rAu^{k+1} - \bar{\lambda}^{k} +rc)],
      \\
      \bar{\lambda}^{k+1} &= \bar{\lambda}^{k} + r(Au^{k+1} + Bp^{k+1}-c). \\
    \end{aligned}
    \right.
  \end{equation}
Henceforth, we assume
\begin{equation}\label{eq:assumption:admm}
  rA^*A + \partial F, \quad rB^*B + \partial G \ \ \text{are strongly maximal monotone},
\end{equation}
such that $  (rA^*A + \partial F)^{-1}$ and $( rB^*B + \partial G)^{-1} $ exist and are Lipschitz continuous.

The convergence of \eqref{admm:original:sys:final} method in finite dimensional spaces is clear; see \cite{DY, HY, LST, MS} and \cite{EP, SM}. However, the weak convergence is not quite clear in infinite dimensional Hilbert space except \cite{BS1}. It is also well known that one can get ADMM through applying Douglas-Rachford splitting method to the dual problems \eqref{eq:saddle-point-prob:dual} while $B=-I$, $c=0$ \cite{DGBA}, and there are some weak convergence results of Douglas-Rachford splitting
method in infinite dimensional Hilbert spaces; see \cite{COM, EP} and the recent paper \cite{BFS}.
However, one can check that the weak convergence of Douglas-Rachford splitting method in \cite{BFS} could not lead to weak convergence of the corresponding ADMM directly. The weak convergence of preconditioned ADMM is analysed in \cite{BS1} in infinite dimensional Hilbert spaces, under the framework of general proximal point iterations with moderate conditions. Here we extend the results to overrelaxed variants through a different approach.

Actually, ADMM can be seen as a special case of the following ``fully" relaxed and preconditioned version of \eqref{admm:original:sys:final},
\begin{equation}
    \label{admm:relax:sys:full:final}\tag{rpFADMM}
    \left\{
    \begin{aligned}
       u^{k+1}
      &= (N + \partial F)^{-1}[A^*(-rBp^k +rc - \bar{\lambda}^k)+(N-rA^*A)u^k],\\
p^{k+1}  &= (M + \partial G)^{-1}[B^*( -r \rho_{k} Au^{k+1} + r(1-\rho_{k})Bp^{k} - \bar{\lambda}^{k} +r\rho_{k}c) + (M-rB^*B)p^k],
      \\
      \bar{\lambda}^{k+1} &= \bar{\lambda}^{k} + r(\rho_{k} Au^{k+1} + Bp^{k+1}-(1-\rho_{k})Bp^{k}-\rho_{k}c),
    \end{aligned}
    \right.
  \end{equation}
  with assumptions on relaxation parameters $\rho_{k}$ throughout this paper,
  \begin{equation}\label{eq:assum:relaxation}
  \rho_{k} \in (0,2), \quad \text{non-decreasing, and} \ \ \lim_{k\rightarrow\infty} \rho_k = \rho^* <2,
  \end{equation}
together with $N \in L(X)$, $M \in L(Y)$ are both self-adjoint and positive definite operators, i.e.,
\begin{equation}\label{eq:assume:M:N}
N - rA^*A \geq 0, \quad M - rB^*B \geq 0.
\end{equation}
We would like to point out that preconditioning for the dual variable $p$ is also necessary in some cases; see the multiplier-based max-flow algorithm \cite{YBT}.
Considering the most simple example, while $\partial G$ and $\partial F$ are linear operators, the updates of $u^{k+1}$ and $p^{k+1}$ in \eqref{admm:relax:sys:full:final} could be reformulated as follows,
\begin{equation*}
    \left\{
    \begin{aligned}
       u^{k+1}
      &= u^{k} + (N + \partial F)^{-1}[A^*(-rBp^k +rc - \bar{\lambda}^k)  - (\partial F+rA^*A)u^k],\\
p^{k+1}  &= p^{k} + (M + \partial G)^{-1}[B^*( -r \rho_{k} Au^{k+1} + r(1-\rho_{k})Bp^{k} - \bar{\lambda}^{k} +r\rho_{k}c) - (\partial G+ rB^*B)p^k],
    \end{aligned}
    \right.
  \end{equation*}
where $(N + \partial F)$ and $(M + \partial G)$ can be seen as generalized ``preconditioners" for the operator equations when calculating $u^{k+1}$ and $p^{k+1}$ in classical \eqref{admm:original:sys:final} with relaxation. It is shown in \cite{BS1} that one could benefit from efficient preconditioners for solving the implicit problems approximately with only one, two, or three cheap preconditioned iterations without controlling the errors, to guarantee the (weak) convergence of the \eqref{admm:original:sys:final} iterations. It is not surprising that while $F$ (or $G$) is a quadratical form \cite{BS1}, the update of $u^{k}$ (or $p^{k}$) above is just the classical preconditioned iteration in numerical linear algebra in finite dimensional spaces, where various efficient preconditioning techniques are widely used. Even for nonlinear $F$ or $G$, there are still some preconditioning techniques available, see \cite{CP1} for diagonal preconditioning, and \cite{SUNDE} for nonlinear symmetric Gauss-Seidel preconditioning.

By direct calculation, it can be shown that \eqref{admm:relax:sys:full:final} is equivalent to the following iterations involving the augmented Lagrangian \eqref{eq:augmented:lag},
\begin{equation}
    \label{admm:relax:sys:full:final:augform}
    \left\{
    \begin{aligned}
      u^{k+1} &:= \argmin_{u \in X} \ L_{r}(u, p^{k}, \bar \lambda^{k},c) + \frac{1}{2} \|u-u^k\|_{N-rA^*A}^2,
      \\
      p^{k+1} &:= \argmin_{p \in Y} \ L_{r}(\rho_k u^{k+1}, p -(1-\rho_k)p^{k}, \bar \lambda^{k}, \rho_k c) + \frac{1}{2} \|p-p^k\|_{M-rB^*B}^2,
      \\
    \bar{\lambda}^{k+1} &:= \bar{\lambda}^{k} + r(\rho_{k} Au^{k+1} + Bp^{k+1}-(1-\rho_{k})Bp^{k}-\rho_{k}c).
    \end{aligned}
    \right.
  \end{equation}
Here $\frac{1}{2} \|u-u^k\|_{N-rA^*A}^2$ and $\frac{1}{2} \|p-p^k\|_{M-rB^*B}^2$ are weighted norm, i.e., $\|u-u^k\|_{N-rA^*A}^2 = \langle u-u^k, (N-rA^*A)(u-u^k) \rangle $, and they are called as proximal terms as in \cite{AS, DY, LST}. \eqref{admm:original:sys:final} could also be recovered from \eqref{admm:relax:sys:full:final:augform} by setting $\rho_k \equiv 1$, $N-rA^*A =0$ and $M-rB^*B=0$.

Let us define the ``partial" primal-dual gap respecting to a point $x$ first, with notation $x = (u,p,\bar \lambda)$ and $x^{k+1} = (u^{k+1},p^{k+1},\bar \lambda^{k+1})$,
\begin{equation}\label{eq:partial:gap:def}
\begin{aligned}
  \gap_{x}(x^{k+1}) &= \gap_{(u, p, \bar \lambda)}(u^{k+1}, p^{k+1}, \bar \lambda^{k+1}): = \mL(u^{k+1}, p^{k+1}, \bar \lambda) - \mL(u, p, \bar \lambda^{k+1}), \\
   & = F(u^{k+1}) + G(p^{k+1}) + \langle \bar  \lambda, Au^{k+1} + Bp^{k+1}-c \rangle  \\
& \quad -[F(u) + G(p) + \langle \bar \lambda^{k+1}  ,  Au +Bp-c\rangle ].
\end{aligned}
\end{equation}
In the sequel, we always assume that $(u,p) \in \dom F \times \dom G$, such that $\mL(u, p, \bar \lambda^{k+1})$ and $\mL(u^{k+1}, p^{k+1}, \bar \lambda)$ are finite, with $\dom F$ (or $\dom G$) defined as
$\dom F:= \{x \in X: F(x) < \infty\}$ (or $\dom G:= \{p \in Y: G(p) < \infty\}$). This assumption can guarantee that the boundedness of $\gap_{x}(x^{k+1})$ while $(u^{k+1},p^{k+1}) \in \dom F \times \dom G$. Similar but different gap function could be found in \cite{UCJ}. Our ``partial" gap function has more complicate structure than that used in \cite{BS3, CP2}.

%

The main contributions belong to the following parts.
First, motivated by the ``partial" primal-dual gap analysis in \cite{BS3} and \cite{CP2}, we prove weak convergence of relaxed \eqref{admm:original:sys:final} in \cite{EP} in infinite dimensional Hilbert space with conditions \eqref{eq:assumption:admm}, \eqref{eq:assum:relaxation} and \eqref{eq:assume:M:N}, by employing the detailed analysis of ``partial" primal-dual gap \eqref{eq:partial:gap:def}.
No additional condition is needed as in \cite{AS} where $N-rA^*A= I$, $M-rB^*B = I$ and $\rho_k \equiv 1$.
 This kind of relaxation is different from \cite{DY, FG, LST} that are focused on relaxation for the updates of Lagrangian multipliers. Second, when $M=rB^*B$ in \eqref{admm:relax:sys:full:final}, while $A$, $B$ and $G$ have similar separable structures, we proposed the scheme that distributing different relaxation parameters for updating different components of $p$ and $\bar \lambda$, which is quite necessary for compact, relaxed and preconditioned ADMM. We also studied the relationship between the relaxed ADMM with preconditioning and the corresponding Douglas-Rachford splitting method which is an extension of \cite{DGBA}. It could help study the properties of the Lagrangian multipliers from the dual formulation.
 The case that $N-rA^*A >0$ and $M=rB^*B$ in \eqref{admm:relax:sys:full:final:augform}, is discussed in \cite{FBX} in finite dimensional spaces. However the positive definiteness of $N-rA^*A$ could prevent lots of interesting applications including the efficient symmetric (block) Red-Black Gauss-Seidel preconditioner for TV (or TGV) denosing problems, where $N-rA^*A$ is not positive definite, see \cite{BS, BS1, BS2}. Similar linearized ADMM with overrelaxation is also considered in \cite{NJ} while $M = rB^*B$, where estimating the largest eigenvalue of $A^*A$ or $B^*B$ is necessary. Third, we prove the weak convergence of ``fully" preconditioned ADMM with relaxation \eqref{admm:relax:sys:full:final} and give the corresponding ergodic convergence rate. To the best knowledge of the author, the weak convergence and ergodic convergence rate of \eqref{admm:relax:sys:full:final} seem to be figured clear for the first time in infinite dimensional Hilbert spaces.

The paper is organized as follows, we first consider the weak convergence and ergodic convergence rate of the relaxation variant of \eqref{admm:original:sys:final} in Section \ref{sec:relax:admm} by employing the ``partial" gap analysis. In Section \ref{sec:pre:admm:relax}, by applying the results for relaxed ADMM in Section \ref{sec:relax:admm} to a modified constrained problem, we get a kind of preconditioned ADMM and its relaxation formally, with only preconditioning for the update of $u^{k+1}$. Furthermore, we present the relation to Douglas-Rachford splitting method, and it turns out that the primal-dual gap function \cite{CP} is useful for ADMM while applying to ROF denoising model. In Section \ref{sec:full:pre:admm}, we discuss the weak convergence and the corresponding ergodic convergence rate of \eqref{admm:relax:sys:full:final}, with adjusted ``partial" gap analysis and the analysis in the previous sections. In the last section, some numerical tests are conducted to demonstrate efficiency of the proposed schemes in this paper.

\section{Relaxation variants of ADMM}\label{sec:relax:admm}
Here, we will consider the following relaxed ADMM method, with assumptions on $\rho_k$ as \eqref{eq:assum:relaxation},
\begin{equation}
    \label{admm:relax:sys:final}\tag{rADMM}
    \left\{
    \begin{aligned}
       u^{k+1}
      &= (rA^*A + \partial F)^{-1}[A^*(-rBp^k +rc - \bar{\lambda}^k)],\\
p^{k+1}  &= (r B^*B + \partial G)^{-1}[B^*( -r \rho_{k} Au^{k+1} + r(1-\rho_{k})Bp^{k} - \bar{\lambda}^{k} +r\rho_{k}c)],
      \\
      \bar{\lambda}^{k+1} &= \bar{\lambda}^{k} + r(\rho_{k} Au^{k+1} + Bp^{k+1}-(1-\rho_{k})Bp^{k}-\rho_{k}c). \\
    \end{aligned}
    \right.
  \end{equation}
We will estimate the ``partial" primal-dual gap, and give a derivation from the iteration \eqref{admm:relax:sys:final} directly. We denote
  \begin{equation}\label{eq:def:po:lambdao}
  \bar \lambda_{o}^{k+1} := \frac{\bar \lambda^{k+1} - \bar \lambda^{k}}{\rho_{k}} + \bar \lambda^{k}, \quad  p_{o}^{k+1} := \frac{p^{k+1} - p^{k}}{\rho_{k}} + p^{k}.
  \end{equation}
By direct calculation,  the last update of \eqref{admm:relax:sys:final} can be reformulated as follows,
  \begin{equation}\label{pre:admm:drway:classical:rewrite}
   \bar \lambda_{o}^{k+1} = \bar \lambda^{k} + r[Au^{k+1} +B p_{o}^{k+1}-c].
\end{equation}
We introduce the following adjacent variables for convenience throughout the paper, for $k \in \mathbb{N}$,
\begin{equation}\label{def:v:w}
\begin{aligned}
&v^{k} := -rBp^{k} + \bar \lambda^{k}, \quad v_{o}^{k+1} := -rBp_{o}^{k+1} + \bar \lambda_{o}^{k+1}, \quad v := -rBp + \bar \lambda, \\
& w^{k} := -rBp^{k} - \bar \lambda^{k}, \quad w_{o}^{k+1} := -rBp_{o}^{k+1} - \bar \lambda_{o}^{k+1}, \quad w: = -rBp - \bar \lambda.
\end{aligned}
\end{equation}
\begin{theorem}\label{thm:gap:estimate:pre:u}
With $v$, $w$ defined in \eqref{def:v:w}, and assuming $Au + Bp  = c$, for the primal-dual gap of the relaxed ADMM \eqref{admm:relax:sys:final}, we have, for $k \geq 0$, while $0< \rho_k <1$,
\begin{equation} \label{eq:gap:estimate:relaxed:rho1}
   \begin{aligned}
  \gap_{x}(x^{k+1})  & \leq
\rho_{k} r \langle B(p^{k} - p_{o}^{k+1}), B(p_{o}^{k+1} - p) \rangle +  \frac{\rho_{k}}{r}\langle \bar \lambda^{k} -\bar \lambda_{o}^{k+1}, \bar \lambda_{o}^{k+1} - \bar \lambda \rangle  \\
       &+ \frac{1-\rho_{k}}{r}\langle v^{k} -v_{o}^{k+1}, v_{o}^{k+1} -v \rangle + \rho_{k}(2-\rho_{k})\langle B(p^{k} - p_{o}^{k+1}), \bar \lambda^{k} -\bar \lambda_{o}^{k+1}\rangle,
  \end{aligned}
\end{equation}
and while $1 \leq \rho_k <2$, we have
   \begin{align}
  \gap_{x }(x^{k+1})  & \leq
(2-\rho_{k}) r \langle B(p^{k} - p_{o}^{k+1}), B(p_{o}^{k+1} - p) \rangle +  \frac{2-\rho_{k}}{r}\langle \bar \lambda^{k} -\bar \lambda_{o}^{k+1}, \bar \lambda_{o}^{k+1} - \bar \lambda \rangle  \label{eq:gap:estimate:relaxed:rho2} \\
       &+\frac{\rho_{k}-1}{r} \langle w^{k} -w_{o}^{k+1}, w_{o}^{k+1} -w \rangle + \rho_{k}(2-\rho_{k})\langle B(p^{k} - p_{o}^{k+1}), \bar \lambda^{k} -\bar \lambda_{o}^{k+1}\rangle. \notag
  \end{align}
\end{theorem}
\begin{proof}
By the update of $u^{k+1}$ in \eqref{admm:relax:sys:final}, we see
\[
A^*(-rBp^k +rc - \bar{\lambda}^k) -rA^*Au^{k+1} \in \partial F(u^{k+1}).
\]
Furthermore, by the update of $\bar \lambda^{k+1}$ in \eqref{admm:relax:sys:final}, we see
\[
-\frac{\bar \lambda^{k+1} - \bar \lambda^{k}}{\rho_{k}} - \bar \lambda^{k} =
-rBp^{k} - rAu^{k+1} +rc -\bar \lambda^{k} - \frac{rB(p^{k+1}-p^{k})}{\rho_{k}},
\]
and what follows is
\[
-\bar \lambda_{o}^{k+1} +rB p_{o}^{k+1} - rBp^{k}  = -rBp^{k} - rAu^{k+1} +rc -\bar \lambda^{k} .
\]
Thus we have
\begin{equation}\label{eq:relax:gap:u}
\langle -\bar \lambda_{o}^{k+1} +rB p_{o}^{k+1} - rBp^{k}, A(u^{k+1}-u) \rangle \in \langle \partial F(u^{k+1}), u^{k+1}-u\rangle.
\end{equation}
By the update of $p^{k+1}$ in \eqref{admm:relax:sys:final}, we see
\[
B^*( -r \rho_{k} Au^{k+1} + r(1-\rho_{k})Bp^{k} - \bar{\lambda}^{k} +r\rho_{k}c) - rB^*Bp^{k+1} \in \partial G(p^{k+1}).
\]
Again by the update of $\bar \lambda^{k+1}$ of \eqref{admm:relax:sys:final}, we have
\begin{equation}\label{eq:relax:p:subgradient}
-r \rho_{k} Au^{k+1} + r(1-\rho_{k})Bp^{k} - \bar{\lambda}^{k} +r\rho_{k}c -r Bp^{k+1} = -\bar \lambda^{k+1} \Rightarrow -B^* \bar \lambda^{k+1} \in \partial G(p^{k+1}),
\end{equation}
and what follows is
\begin{equation}\label{eq:relax:gap:p}
\langle-B^* \bar \lambda^{k+1}, p^{k+1}  - p \rangle \in \langle \partial G(p^{k+1}), p^{k+1}  - p \rangle.
\end{equation}
By the definition of subgradient, we have
\begin{align}
& F(u^{k+1}) - F(u) \leq \langle \zeta^{k+1}, u^{k+1} -u \rangle, \quad \forall \zeta^{k+1} \in \partial F(u^{k+1}), \label{eq:sub:gra:f} \\
& G(p^{k+1}) - G(p) \leq  \langle  \eta^{k+1}, p^{k+1} -p \rangle,  \quad \forall \eta^{k+1} \in \partial G(p^{k+1}). \label{eq:sub:gra:g}
\end{align}
With \eqref{eq:relax:gap:u}, \eqref{eq:relax:gap:p}, \eqref{eq:sub:gra:f}, \eqref{eq:sub:gra:g} and the assumption $Au+Bp=c$, for the right hand side of \eqref{eq:partial:gap:def}, we see
\[
  \gap_{x}(x^{k+1})  \leq \text{I} + \text{II}.
\]
Here $\text{I}$ and $\text{II}$ are as follows,
\begin{equation}\label{eq:I}
\text{I} = \langle -rBp^{k} - \bar \lambda_{o}^{k+1} +rBp_{o}^{k+1}, A(u^{k+1} - u) \rangle + \langle -\bar \lambda_{o}^{k+1}, B(p_o^{k+1} - p) \rangle + \langle \bar \lambda, Au^{k+1}+Bp_o^{k+1}-c \rangle,
\end{equation}
\begin{align*}\label{eq:II}
\text{II} &= \langle -B^*\bar  \lambda^{k+1}, p^{k+1}-p \rangle + \langle  \bar \lambda, Au^{k+1}+Bp^{k+1}-c \rangle \\
&\quad + \langle \bar \lambda_o^{k+1}, B(p_o^{k+1}-p) \rangle  - \langle \bar \lambda, Au^{k+1}+Bp_o^{k+1}-c \rangle \\
& =    \langle \bar  \lambda^{k+1}, -B(p^{k+1}-p) \rangle + \langle \bar \lambda_o^{k+1}, B(p_o^{k+1}-p) \rangle+ \langle \bar \lambda, B(p^{k+1}-p_o^{k+1}) \rangle \\
&  =\langle \bar  \lambda^{k+1}, -B(p^{k+1}-p) \rangle + \langle \bar \lambda_o^{k+1}, B(p_o^{k+1}-p) \rangle
+ \langle \bar \lambda, B(p^{k+1}-p)+B(p-p_o^{k+1}) \rangle, \\
& = \langle -B(p^{k+1}-p), \bar \lambda^{k+1} - \bar \lambda  \rangle + \langle  B(p_o^{k+1}-p), \bar \lambda_o^{k+1} - \bar \lambda \rangle.
\end{align*}
By \eqref{pre:admm:drway:classical:rewrite}, we see,
\[
Au^{k+1} + Bp_{o}^{k+1} -c = (\bar \lambda_o^{k+1} - \bar \lambda^{k})/{r},
\]
and together with the assumption $Au+Bp-c = 0$, we have
\[
A(u^{k+1}-u) = (\bar \lambda_o^{k+1} - \bar \lambda^{k})/{r} - B(p_o^{k+1}-p).
\]
Substituting these two equalities into I, we have
\begin{equation}\label{eq:ref:I}
\text{I} = \frac{1}{r}\langle  \bar \lambda^{k} -\bar \lambda_o^{k+1}, \bar \lambda_o^{k+1} - \bar  \lambda  \rangle + r\langle B(p^k - p_o^{k+1}), B(p_o^{k+1}-p) \rangle - \langle p_o^{k+1} - p^k, -B^*(\bar \lambda_o^{k+1}  -\bar \lambda^{k}) \rangle.
\end{equation}
By \eqref{eq:def:po:lambdao}, we see
\begin{equation}\label{eq:p:lambda:po:lambdao}
p^{k+1} = \rho_{k}p_o^{k+1} +(1-\rho_{k})p^{k}, \quad \bar \lambda^{k+1} = \rho_{k} \bar \lambda_o^{k+1} +(1-\rho_{k})\bar \lambda^{k}.
\end{equation}
Substituting \eqref{eq:p:lambda:po:lambdao} into II, we see
\begin{align}
\text{II} =& -(1-\rho_{k}) \langle B(p^{k} - p_o^{k+1}),\bar \lambda_o^{k+1} - \bar \lambda \rangle -(1-\rho_{k}) \langle B(p_o^{k+1} -p), \bar \lambda^{k} - \bar \lambda_o^{k+1} \rangle \notag \\
& - (1-\rho_{k})^2\langle B(p^k -p_o^{k+1}), \bar \lambda^{k} - \bar \lambda_o^{k+1} \rangle. \label{eq:II:final:esti}
\end{align}
Combining \eqref{eq:I}, \eqref{eq:ref:I} and \eqref{eq:II:final:esti}, we have
\begin{align}
 \gap_{x}(x^{k+1})  &\leq \text{I} + \text{II}= \frac{1}{r}\langle  \bar \lambda^{k} -\bar \lambda_o^{k+1}, \bar \lambda_o^{k+1} - \bar \lambda  \rangle + r\langle B(p^k - p_o^{k+1}), B(p_o^{k+1}-p) \rangle \notag \\
& + \rho_{k} (2-\rho_{k}) \langle B(p^k-p_o^{k+1}), \bar\lambda^{k} -\bar \lambda_o^{k+1} \rangle \label{eq:gap:estimate:demtalil} \\
& -(1-\rho_k) \langle B(p^k-p_o^{k+1}), \bar \lambda_o^{k+1} -\bar \lambda \rangle -(1-\rho_k) \langle B(p_o^{k+1} -p), \bar \lambda^k - \bar \lambda_o^{k+1} \rangle.   \notag
\end{align}
For last two mixed terms in \eqref{eq:gap:estimate:demtalil}, we notice that
\begin{align}
\frac{1}{r}\langle v^{k} - v_{o}^{k+1}, v_{o}^{k+1} - v \rangle &= r \langle B(p^{k} - p_{o}^{k+1}), B(p_{o}^{k+1} - p) \rangle +  \frac{1}{r}\langle \bar \lambda^{k} -\bar \lambda_{o}^{k+1}, \bar \lambda_{o}^{k+1} - \bar \lambda \rangle  \notag \\
&+ \langle  -B(p^{k} - p_{o}^{k+1}),  \bar \lambda_{o}^{k+1} - \bar \lambda \rangle +  \langle  -B(p_{o}^{k+1}-p),  \bar \lambda^{k}- \bar  \lambda_{o}^{k+1} \rangle, \label{eq:expan:v} \\
\frac{1}{r}\langle w^{k} - w_{o}^{k+1}, w_{o}^{k+1} - w \rangle &= r \langle B(p^{k} - p_{o}^{k+1}), B(p_{o}^{k+1} - p) \rangle +  \frac{1}{r}\langle \bar \lambda^{k} -\bar \lambda_{o}^{k+1}, \bar \lambda_{o}^{k+1} - \bar \lambda \rangle  \notag \\
&+ \langle  B(p^{k} - p_{o}^{k+1}),  \bar \lambda_{o}^{k+1} - \bar \lambda \rangle +  \langle  B(p_{o}^{k+1}-p),  \bar \lambda^{k}- \bar  \lambda_{o}^{k+1} \rangle. \label{eq:expan:w}
\end{align}
Thus if $0 < \rho_{k} <1$, we can replace the last two mixed terms in \eqref{eq:gap:estimate:demtalil} by linear combination of the first two terms of the right hand side of \eqref{eq:expan:v} and  the left hand side of \eqref{eq:expan:v}, i.e., $1/r\langle v^{k} - v_{o}^{k+1}, v_{o}^{k+1} - v \rangle$,
leading to the estimate \eqref{eq:gap:estimate:relaxed:rho1}. Similarly, if $1 \leq \rho_{k} <2$, by \eqref{eq:expan:w} and \eqref{eq:gap:estimate:demtalil}, we get the estimate \eqref{eq:gap:estimate:relaxed:rho2}.
\end{proof}
\begin{remark}
If $\rho_k \equiv 1$ without relaxation, we have the simpler estimte. Assuming $Au+ Bp = c$, considering iteration \eqref{admm:original:sys:final}, regarding to $(u, p, \bar \lambda)$, we have the following partial primal dual gap estimate for $k \geq 0$,
\begin{equation*} 
   \begin{aligned}
  \gap_{(u,p,\bar \lambda)}(u^{k+1}, p^{k+1}, \bar \lambda^{k+1})  & \leq
r \langle B(p^{k} - p^{k+1}),B( p^{k+1} - p) \rangle +  \frac{1}{r}\langle \bar \lambda^{k} -\bar \lambda^{k+1}, \bar \lambda^{k+1} - \bar \lambda \rangle  \\
       & \quad - \langle p^{k+1} - p^{k}, -B^*(\bar \lambda^{k+1} -\bar \lambda^{k})\rangle.
  \end{aligned}
\end{equation*}
\end{remark}
By Theorem \ref{thm:gap:estimate:pre:u} and \eqref{eq:def:po:lambdao}, we have the following gap estimate.
\begin{corollary}\label{cor:relax:partial}
For iteration \eqref{admm:relax:sys:final}, with $v$, $w$ defined in \eqref{def:v:w}, and assuming $Au + Bp  = c$, for the primal-dual gap of \eqref{admm:relax:sys:final}, we have, for $k \geq 0$, while $0< \rho_k <1$,
   \begin{align}
  &\gap_{x}(x^{k+1})   \leq
 \frac{r}{2}(\|B(p^{k}-p)\|^2-\|B(p^{k+1}-p)\|^2) +  \frac{1}{2r}(\|\bar \lambda^{k} -\bar \lambda \|^2 -\|\bar\lambda^{k+1} -\bar \lambda \|^2 ) \notag \\
       &+ \frac{1-\rho_{k}}{2\rho_{k}r}(\|v^{k}-v\|^2-\|v^{k+1}-v\|^2) - \frac{2-\rho_{k}}{\rho_{k}}\langle -B(p^{k} - p^{k+1}), \bar \lambda^{k} -\bar \lambda^{k+1}\rangle \label{eq:gap:estimate:relaxed:rho1:origin}\\
       &- \frac{r(2-\rho_{k})}{2\rho_{k}}\|B(p^{k}-p^{k+1})\|^2-\frac{(2-\rho_{k})}{2r\rho_{k}}\|\bar\lambda^{k}-\bar\lambda^{k+1}\|^2
       - \frac{1-\rho_{k}}{r}\frac{2-\rho_{k}}{2\rho_{k}^2}\|v^{k}-v^{k+1}\|^2. \notag
  \end{align}
While $1 \leq \rho_k <2$, we have
   \begin{align}
  &\gap_{x}(x^{k+1})   \leq
 \frac{r(2-\rho_{k})}{2\rho_{k}}(\|B(p^{k}-p)\|^2-\|B(p^{k+1}-p)\|^2) +  \frac{2-\rho_{k}}{2r\rho_{k}}(\|\bar \lambda^{k} -\bar \lambda \|^2 -\|\bar\lambda^{k+1} -\bar \lambda \|^2 ) \notag \\
       &+ \frac{\rho_{k}-1}{2\rho_{k}r}(\|w^{k}-w\|^2-\|w^{k+1}-w\|^2) - \frac{2-\rho_{k}}{\rho_{k}}\langle -B(p^{k} - p^{k+1}), \bar \lambda^{k} -\bar \lambda^{k+1}\rangle \label{eq:gap:estimate:relaxed:rho2:origin}\\
       &- \frac{r(2-\rho_{k})^2}{2\rho_{k}^2}\|B(p^{k}-p^{k+1})\|^2-\frac{(2-\rho_{k})^2}{2r\rho_{k}^2}\|\bar\lambda^{k}-\bar\lambda^{k+1}\|^2
       - \frac{\rho_{k}-1}{r}\frac{2-\rho_{k}}{2\rho_{k}^2}\|w^{k}-w^{k+1}\|^2. \notag
  \end{align}
\end{corollary}
\begin{proof}
We just need to show the following three inner products, since others are similar. By \eqref{eq:def:po:lambdao} and polarization identity, we arrive at
\begin{align*}
\langle B(p^{k}-p_{o}^{k+1}), B(p_{o}^{k+1} - p) \rangle & = \langle B (p^{k}-p^{k+1})/\rho_{k}, B\{(p^{k+1}-p^{k})/\rho_{k} +p^{k}- p^{k+1}+ p^{k+1} - p\} \rangle \\
& = (-\frac{1}{\rho_{k}^2}+ \frac{1}{\rho_{k}})\|B(p^{k}-p^{k+1})\|^2 + \frac{1}{\rho_{k}}\langle B( p^{k}-p^{k+1}), B(p^{k+1} - p) \rangle\\
& =  \frac{1}{2\rho_{k}}(\|B(p^{k}-p)\|^2-\|B(p^{k+1}-p)\|^2)-\frac{2-\rho_{k}}{2\rho_{k}^2}\|B(p^{k}-p^{k+1})\|^2.
\end{align*}
Similarly, we have
\begin{align*}
\langle \bar \lambda^{k} -\bar \lambda_{o}^{k+1}, \bar \lambda_{o}^{k+1} - \bar \lambda \rangle
& = \frac{1}{2 \rho_k}(\| \bar \lambda^{k}-\bar \lambda\|^2 -\| \bar \lambda^{k+1}-\bar \lambda^k\|^2) - \frac{2-\rho_k}{2 \rho_k^2}\|\bar \lambda^{k+1} - \bar \lambda^k\|^2.
\end{align*}
Consequently, we get
\begin{align*}
 \rho_{k}(2-\rho_{k})\langle -B(p^{k} - p_{o}^{k+1}), \bar \lambda^{k} -\bar \lambda_{o}^{k+1}\rangle &=  \rho_{k}(2-\rho_{k})\langle -B(p^{k} - p^{k+1})/\rho_{k}, (\bar \lambda^{k} -\bar \lambda^{k+1})/\rho_{k}\rangle\\
 &=\frac{2-\rho_{k}}{\rho_{k}}\langle -B( p^{k} - p^{k+1}), \bar \lambda^{k} -\bar \lambda^{k+1}\rangle.
\end{align*}

The cases for $ \langle  v^{k} -v_{o}^{k+1}, v_{o}^{k+1} - v \rangle $ and $ \langle w^{k} -w_{o}^{k+1}, w_{o}^{k+1} - w \rangle $ are similar.
Substituting these formulas to Theorem \ref{thm:gap:estimate:pre:u}, we get this corollary.
\end{proof}
\begin{lemma}
  \label{lem:admm-fixed-points:rela}
  Iteration~\eqref{admm:relax:sys:final} possesses the following
  properties:
  \begin{enumerate}
  \item
    A point $(u^*,p^*,\bar\lambda^* )$ is a fixed point if and only
    if $(u^*, p^*, \bar \lambda^*)$ is a saddle point for~\eqref{eq:saddle-point-prob}.
  \item
      If $\wlim_{i \to \infty}(u^{k_i}, p^{k_i}, \bar \lambda^{k_i}) = (u^*, p^*,\bar \lambda^*)$ and
    $\lim_{i \to \infty} ( B(p^{k_i} -  p^{k_i+1}), \bar \lambda^{k_i} - \bar \lambda^{k_i+1})
    = (0,0)$, then
    $(u^{k_i+1}, p^{k_i+1}, \bar \lambda^{k_i+1})$ converges weakly to
    a fixed point, where $\wlim$ is denoted as the weak convergence in Hilbert spaces hereafter.
  \end{enumerate}
\end{lemma}
\begin{proof}
Regarding to the first point, $(u^*,p^*, \bar \lambda^*)$ is a saddle point for~\eqref{eq:saddle-point-prob} is equivalent to
\begin{equation}\label{eq:saddle:con}
0 \in \partial F(u^*) + A^* \bar \lambda^*, \quad 0 \in B^* \bar \lambda^* + \partial G(p), \quad Au^* + Bp^* -c = 0 .
\end{equation}
A point $(u^*,p^*,\bar\lambda^* )$
is a fixed point of iteration \eqref{admm:relax:sys:final}, by \eqref{eq:assum:relaxation}, if and only if
\begin{subequations}
\begin{align}\label{eq:opimi:fixed:points}
& A^*(-rBp^* - \bar \lambda^* +rc) -rA^*Au^{*} \in \partial F(u^*), \\
&B^*(-r\rho^*(Au^{*} +Bp^* -c) - \bar \lambda^* ) \in \partial G(p^{*}), \\
&\rho^*(Au^* + Bp^* -c)  =0.\label{eq:opimi:fixed:points:last}
\end{align}
\end{subequations}
Since $\rho^* \in (0,2)$ and $r >0$,
thus \eqref{eq:opimi:fixed:points}-\eqref{eq:opimi:fixed:points:last} are equivalent to
\begin{equation}\label{eq:fixed:con}
-A^* \bar \lambda^* \in \partial F(u^*), \quad  -B^*\bar \lambda^* \in \partial G(p^*), \quad Au^* + Bp^* -c  =0,
\end{equation}
which means that $(u^*, p^*, \bar \lambda^*)$ is a saddle-point of \eqref{eq:saddle-point-prob}.

Next, by the updates of $u^{k+1}$ and $p^{k+1}$ in \eqref{admm:relax:sys:final}, we see that
\begin{align*}
& A^* \bar \lambda^{k+1} + A^*(-rBp^k - \bar \lambda^k +rc) -rA^*Au^{k+1} \in \partial F(u^{k+1}) + A^*\bar \lambda^{k+1}, \\
&B^* \bar \lambda^{k+1}+  B^*(-r\rho_k Au^{k+1} +r(1-\rho_k)Bp^k - \bar \lambda^k +r\rho_{k}c) -rB^*B p^{k+1} \in \partial G(p^{k+1})+ B^* \bar \lambda^{k+1},
\end{align*}
and by the update of $\bar \lambda^{k_i+1}$  in \eqref{admm:relax:sys:final} and the second assumption, we have
\[
Au^{k_i+1} + Bp^{k_i} - c \rightarrow 0.
\]
What follows is
\begin{align}
 A^*(\bar \lambda^{k_i+1} -\bar \lambda^{k_i})-rA^*(Au^{k_i+1} + Bp^{k_i} - c)  &\in \partial F(u^{k_i+1}) + A^*\bar \lambda^{k_i+1} \rightarrow 0, \\
0  &\in \partial G(p^{k_i+1})+ B^* \bar \lambda^{k_i+1} \rightarrow 0, \label{eq:opti:p}\\
Au^{k_i+1} + Bp^{k_i} - c + B(p^{k_i+1}-p^{k_i}) &= Au^{k_i+1} + Bp^{k_i+1}-c \rightarrow 0.
\end{align}

Combining the weak convergence of $(u^{k_i}, p^{k_i}, \bar \lambda^{k_i})$ as in the assumption, and the maximally monotone of
\[
  (u,p, \bar \lambda) \mapsto \bigl(A^* \bar \lambda + \subgrad F(u), B^* \bar \lambda
  + \subgrad G(p), -Au-Bp+c \bigr),
\]
which is also weak-strong closed (see Proposition 20.33 of \cite{HBPL}), we have
  $(0,0,0) \in \bigl(A^* \bar \lambda^* + \subgrad F(u^*),
  B^* \bar \lambda^* + \subgrad G(p), -Au^*-B^*p+c \bigr)$
  and $(u^*,p^*, \bar \lambda^*)$ is a saddle-point of~\eqref{saddle:primal:original}.
\end{proof}

%
\begin{theorem}\label{thm:gap:estimate:pre:u:weak}
With the assumptions \eqref{eq:assumption:admm} and \eqref{eq:assum:relaxation} on the relaxation parameters $\rho_k$, if \eqref{saddle:primal:original} possesses a solution, then the iteration sequences $(u^k, p^k, \bar \lambda^k)$ of \eqref{admm:relax:sys:final} converge weakly to a saddle-point $(u^*, p^*, \bar \lambda^*)$ of \eqref{eq:saddle-point-prob} with $(u^*,p^*)$ being a solution of \eqref{saddle:primal:original}.
\end{theorem}
\begin{proof}
Here we mainly prove the boundedness of the iteration sequence $(u^k, p^k, \bar \lambda^k)$ of \eqref{admm:relax:sys:final}, and the uniqueness of the weak limit for all weakly convergent subsequences of $(u^k, p^k, \bar \lambda^k)$. These lead to the weak convergence of $(u^k, p^k, \bar \lambda^k)$.
Remembering that $-B^* \bar \lambda^{k+1} \in \partial G(p^{k+1})$ by \eqref{eq:relax:p:subgradient}, we have
\begin{equation}\label{eq:estimate:relax:rho:ine:inclu:g}
- \frac{2-\rho_{k}}{\rho_{k}}\langle -B(p^{k} - p^{k+1}), \bar \lambda^{k} -\bar \lambda^{k+1}\rangle \leq 0, \quad \forall k \geq 1.
\end{equation}
Let us consider the case $\rho_k \in (0,1) $ first.
Summing up the inequality \eqref{eq:gap:estimate:relaxed:rho1:origin} from $k_0$ to $k$, $k > k_0\geq 1$, with \eqref{eq:estimate:relax:rho:ine:inclu:g}, non-increasing of $\frac{1-\rho_k}{2\rho_kr}$ by \eqref{eq:assum:relaxation}, and notation,
\[
  d_{k,\rho_k} :=\frac{r\norm{ B(p^{k} - p)}^2}{2}
    + \frac{\norm{\bar \lambda^{k} - \lambda}^2}{2r}+ \frac{1-\rho_{k}}{2\rho_{k}r} \|v^{k} -v\|^2,
\]
we have
  \begin{equation}\label{eq:dr_boundedness:relax}
   d_{k,\rho_k}
    +     \sum_{k' = k_0}^{k-1}\frac{2-\rho_{k'}}{2\rho_{k'}}\biggl[
 \frac{r\norm{ B(p^{k'} - p^{k'+1})}^2}{2}
  +  \frac{\norm{\bar \lambda^{k'} - \bar \lambda^{k'+1}}^2}{2r} + \frac{1-\rho_{k'}}{\rho_{k'} r} \|v^{k'}-v^{{k'}+1}\|^2 \biggr]  \leq
     d_{k_0,\rho_{k_0}}.
  \end{equation}

\eqref{eq:dr_boundedness:relax} implies that $\seq{d_{k,\rho_k}}$
  is a non-increasing sequence with limit $d^*$. Furthermore,
  $B(p^k - p^{k+1}) \to 0$ as well as $\bar \lambda^k - \bar \lambda^{k+1} \to 0$
  as $k \to \infty$. Since $\frac{2-\rho_k}{2\rho_k} \geq \frac{1}{2}$ and $\frac{1-\rho_k}{2\rho_kr}  \geq 0$, \eqref{eq:dr_boundedness:relax} also tells that $\seq{(Bp^k,\bar \lambda^k)}$ is uniformly
  bounded whenever~\eqref{saddle:primal:original} has a solution.

  By the Lipschitz continuity of $(rA^*A + \partial F)^{-1}$, $(rB^*B+\partial G)^{-1}$ and the updates of $u^{k+1}$, $p^{k+1}$ in \eqref{admm:relax:sys:final}, we can get the uniform boundedness of $(u^{k}, p^{k}, \bar \lambda^{k})$ of iteration \eqref{admm:relax:sys:final}.
By the Banach-Alaoglu Theorem, there exists a weakly convergent subsequence of $(u^k, p^k, \bar \lambda^k)$ and we denote the weak limit of the subsequence as
  $( u^*,p^*, \bar \lambda^*)$. By virtue of Lemma~\ref{lem:admm-fixed-points:rela},
  $(u^*, p^*, \bar \lambda^*)$ is a fixed-point of the iteration \eqref{admm:relax:sys:final}.
 Let us prove the uniqueness of the weak limit.
  Supposing that $(u^{**}, p^{**}, \bar \lambda^{**})$ is another
  weak limit of another subsequence,
  then we have,
  \begin{align*}
   &r \scp{ Bp^k}{ B(p^{**} - p^{*})} +
    \frac{1}{r} \scp{\bar \lambda^k}{\bar \lambda^{**} - \bar \lambda^{*}} +  \frac{1-\rho_k}{\rho_kr} \scp{v^k}{v^{**} - v^{*}} =
      \frac{r\norm{B(p^k - p^*) }^2}{2}
      +   \frac{\norm{\bar \lambda^k - \lambda^*}^2}{2r}
    \\
    &
      + \frac{1-\rho_{k}}{2\rho_{k}r} \|v^{k} -v^*\|^2-\frac{r\norm{B(p^k - p^{**})}^2}{2}
      -   \frac{\norm{\bar \lambda^k - \lambda^{**}}^2}{2r} -\frac{1-\rho_{k}}{2\rho_{k}r} \|v^{k} -v^{**}\|^2
   \\
    &   - \frac{r\norm{Bp^*}^2}{2} - \frac{\norm{\bar \lambda^*}^2}{2r}
      + \frac{r\norm{B p^{**}}^2}{2}
      + \frac{\norm{\bar \lambda^{**}}^2}{2r} - \frac{1-\rho_k}{\rho_kr}\|v^*\|^2 + \frac{1-\rho_k}{2\rho_kr}\|v^{**}\|^2.
  \end{align*}
  Let us take the subsequences $(p^{k_i}, \bar \lambda^{k_i})$ and $(p^{k_j}, \bar \lambda^{k_j})$ in the above equation, which have weak limits $(p^*, \bar \lambda^*)$ and $(p^{**}, \bar \lambda^{**})$ correspondingly. We have,
    \[
    r \scp{ Bp^{**} - Bp^*}{ B(p^{**} - p^{*})} +
    \frac{1}{r} \scp{\bar \lambda^{**} - \bar \lambda^{*}}{\bar \lambda^{**} - \bar \lambda^{*}} +  \frac{1-\rho^*}{\rho^{*}r} \scp{v^{**} - v^{*}}{v^{**} - v^{*}} = 0,
    \]
    and what follows are $Bp^* = Bp^{**}$ and $\bar \lambda^* = \bar \lambda^{**}$.

 Now we prove there is only one weak limit for all weakly convergent subsequences of $(u^k, p^k, \bar \lambda^k)$. By Lemma \ref{lem:admm-fixed-points:rela}, $(u^*, p^*, \bar \lambda^*)$ and $(u^{**}, p^{**}, \bar \lambda^{**})$ are both fixed points of the iteration \eqref{admm:relax:sys:final}. Thus, they both satisfy
  \begin{equation}
    \label{admm:fxied:points}
    \left\{
    \begin{aligned}
       &u'
      = (rA^*A + \partial F)^{-1}[A^*(-rBp' +rc - \bar{\lambda}')],\\
&p'  = (r B^*B + \partial G)^{-1}[B^*( rBp' - \bar{\lambda}')],
      \\
      &Au' + Bp'=c.
    \end{aligned}
    \right.
  \end{equation}
  By the Lipschitz continuity of $(rA^*A + \partial F)^{-1}$ and the first equation of \eqref{admm:fxied:points}, we have
  \[
  \|u^*-u^{**}\| \leq \|(rA^*A + \partial F)^{-1}\|(r\|A^*\| \|Bp^*-Bp^{**}\|+ \|A^*\|\|\bar \lambda^* - \bar \lambda^{**}\|)= 0,
  \]
   which leads to $u^* = u^{**}$. Similarly, by the Lipschitz continuity of $(rB^*B + \partial G)^{-1}$ and second equation of \eqref{admm:fxied:points}, we have $p^* = p^{**}$, which is based on
  \[
    \|p^*-p^{**}\| \leq \|(rB^*B + \partial G)^{-1}\|(r\|B^*\|(\|B\| \|p^*-p^{**}\|+ \|\bar \lambda^* - \bar \lambda^{**}\|))= 0,
  \]
which leads to the uniqueness of the weak limits of all weakly convergent subsequences of $(u^k, p^k, \bar \lambda^k)$. Thus we prove the iteration sequence $(u^k, p^k, \bar \lambda^k)$ of \eqref{admm:original:sys:final} weakly converge to $(u^*, p^*, \bar \lambda^*)$.

%

    For $\rho_k \in [1,2)$, 
    if $\rho_{k} \equiv  c_0 \in(1,2)$ being a constant, the arguments are completely similar to $\rho_k \in (0,1)$ case with \eqref{eq:gap:estimate:relaxed:rho2:origin}.
       However, for varying $\rho_k$,  we need more analysis since $\frac{2-\rho_k}{2\rho_k}$ and $\frac{\rho_k-1}{2\rho_k}$ in \eqref{eq:gap:estimate:relaxed:rho2:origin} have different monotonicity. By the definition of $w^{k}$ in \eqref{def:v:w}, with $x=(u,p,\bar \lambda)$ being a solution of \eqref{eq:saddle-point-prob}, we have
\begin{equation} \label{eq:expan:w:relax}
\begin{aligned}
&\frac{\rho_{k}-1}{2\rho_{k}r}(\|w^{k}-w\|^2-\|w^{k+1}-w\|^2) \\
&= \frac{\rho_{k}-1}{2\rho_{k}r} (\| rBp^k - rBp + \bar \lambda^k - \bar \lambda\|^2 - \| rBp^{k+1} - rBp + \bar \lambda^{k+1} - \bar \lambda\|^2) \\
&= \frac{r(\rho_{k}-1)}{2\rho_{k}}(\| Bp^k - Bp\|^2 -\| Bp^{k+1} - Bp\|^2 )  + \frac{\rho_{k}-1}{2\rho_{k}r} (\|\bar \lambda^k - \bar \lambda\|^2 - \|\bar \lambda^{k+1} - \bar \lambda\|^2) \\
& \quad  + \frac{\rho_{k}-1}{\rho_{k} }[-\langle p^k-p, -B^*(\bar \lambda^k -\bar \lambda) \rangle + \langle p^{k+1}-p, -B^*(\bar \lambda^{k+1} -\bar \lambda) \rangle].
\end{aligned}
\end{equation}
Substituting \eqref{eq:expan:w:relax} into \eqref{eq:gap:estimate:relaxed:rho2:origin}, we have for $1 \leq \rho_k <2$,
\begin{equation}\label{ieq:relax:rho:2}
   \begin{aligned}
  \gap_{x}(x^{k+1})   &\leq
 \frac{r}{2\rho_{k}}(\|B(p^{k}-p)\|^2-\|B(p^{k+1}-p)\|^2) +  \frac{1}{2r\rho_{k}}(\|\bar \lambda^{k} -\bar \lambda \|^2 -\|\bar\lambda^{k+1} -\bar \lambda \|^2 ) \\
       &+ \frac{\rho_{k}-1}{\rho_{k} }[-\langle p^k-p, -B^*(\bar \lambda^k -\bar \lambda) \rangle + \langle p^{k+1}-p, -B^*(\bar \lambda^{k+1} -\bar \lambda) \rangle] \\
       &- \frac{r(2-\rho_{k})}{2\rho_{k}^2}\|B(p^{k}-p^{k+1})\|^2-\frac{(2-\rho_{k})}{2r\rho_{k}^2}\|\bar\lambda^{k}-\bar\lambda^{k+1}\|^2
       - \frac{\rho_{k}-1}{r}\frac{2-\rho_{k}}{2\rho_{k}^2}\|w^{k}-w^{k+1}\|^2 \\
       & - \frac{2-\rho_{k}}{\rho_{k}}\langle -B(p^{k} - p^{k+1}), \bar \lambda^{k} -\bar \lambda^{k+1}\rangle.
  \end{aligned}
\end{equation}
  Let us denote
  \begin{equation}\label{eq:inequality:relax:rho2:trans}
\begin{aligned}
& d_k(x^{k},x): = \frac{r\norm{ B( p^{k} - p)}^2}{2 \rho_k}
    + \frac{\norm{\bar \lambda^{k} - \lambda}^2}{2\rho_k r}  -\frac{\rho_{k}-1}{\rho_{k} }\langle p^k-p, -B^*(\bar \lambda^k -\bar \lambda) \rangle \\
     & =\frac{r\norm{ B( p^{k} - p)}^2}{2 \rho_k}
    + \frac{\norm{\bar \lambda^{k} - \lambda}^2}{2\rho_k r}  +\frac{1}{\rho_{k} }\langle p^k-p, -B^*(\bar \lambda^k -\bar \lambda) \rangle - \langle p^k-p, -B^*(\bar \lambda^k -\bar \lambda) \rangle.
\end{aligned}
\end{equation}
By the assumption that $x=(u,p,\bar \lambda)$ being a solution of \eqref{eq:saddle-point-prob}, we have $-B^*\bar \lambda \in \partial G(p)$, and similar to the inequality \eqref{eq:estimate:relax:rho:ine:inclu:g}, we get
\begin{equation}\label{eq:fight:back}
\langle p^k-p, -B^*(\bar \lambda^k -\bar \lambda) \rangle \geq 0.
\end{equation}
Summing up the inequality \eqref{ieq:relax:rho:2} from $k_0$ to $k$, $k> k_0 \geq 1$, with \eqref{eq:estimate:relax:rho:ine:inclu:g}, \eqref{eq:inequality:relax:rho2:trans}, \eqref{eq:fight:back}, the non-decreasing of $\rho_{k}$ and the cancellation of the terms $-\langle p^k-p, -B^*(\bar \lambda^k -\bar \lambda) \rangle$ from $k_0+1$ to $k-1$, we have
 \begin{align}
   & d_k(x^{k},x) +     \sum_{k' = k_0}^{k-1}\frac{(2-\rho_{k'})^2}{2\rho_{k'}^2}\biggl[
 \frac{r\norm{ B(p^{k'} - p^{k'+1})}^2}{2}
    +  \frac{\norm{\bar \lambda^{k'} - \bar \lambda^{k'+1}}^2}{2r} + \frac{\rho_{k'} -1}{(2-\rho_{k'}) r} \|w^{k'}-w^{k'+1}\|^2 \biggr] \notag \\
& \leq d_{k_0}(x^{k_0},x). \label{eq:inequality:relax:rho2}
  \end{align}
 We conclude that $d_k(x^k,x)$ is a decreasing sequence by \eqref{eq:inequality:relax:rho2}. Similarly, the uniform boundedness of $(Bp^{k}, \bar \lambda^k)$ could be obtained by \eqref{eq:inequality:relax:rho2} and \eqref{eq:inequality:relax:rho2:trans}. Then we get the uniform boundedness of the iteration sequence $x^{k} = (u^{k},p^{k}, \bar \lambda^{k})$, by iteration \eqref{admm:relax:sys:final} with Lipschitz continuity of $(\partial F + rA^*A)^{-1}$ and $(\partial G + rB^*B)^{-1}$. Supposing there exist two subsequence $x^{k_i}$ and $x^{k_j}$ of $x^{k}$ that weakly converge to two weak limits $x^* = (u^*, p^*, \bar \lambda^*)$ and $x^{**} = (u^{**}, p^{**}, \bar \lambda^{**})$ separately,
 by direct calculation, we have
\begin{equation}\label{eq:relax:limit:unique:primal}
\begin{aligned}
&\frac{r}{\rho_k} \langle Bp^k, Bp^{**} - Bp^* \rangle  + \frac{\rho_{k}-1}{\rho_{k} }\{\langle p^k-p^{**}, -B^*(\bar \lambda^k -\bar \lambda^{**}) \rangle  -\langle p^k-p^*, -B^*(\bar \lambda^k -\bar \lambda^*) \rangle\}\\
 & + \frac{1}{r\rho_k} \langle \bar \lambda^k, \bar \lambda^{**} - \bar \lambda^* \rangle = d_k(x^{k},x^{*}) - d_k(x^{k},x^{**}) -\frac{r\|Bp^*\|^2}{2\rho_k}+\frac{r\|Bp^{**}\|^2}{2\rho_k} -\frac{\|\bar \lambda^*\|}{2 \rho_k r}+\frac{\|\bar \lambda^{**}\|}{2 \rho_k r}.
\end{aligned}
\end{equation}
In order to avoiding the subtle issue of the convergence properties of the inner product of two weakly convergent sequences, we need to rewrite the inner product terms,
\begin{equation}\label{eq:inequality:relax:lambda:rho2}
\begin{aligned}
 &\langle p^k-p^{**}, -B^*(\bar \lambda^k -\bar \lambda^{**}) \rangle
  -\langle p^k-p^*, -B^*(\bar \lambda^k -\bar \lambda^*) \rangle \\
 & =  \langle p^k- p^* + p^* - p^{**}, -B^*(\bar \lambda^k -\bar \lambda^{**}) \rangle
 -\langle p^k-p^*, -B^*(\bar \lambda^k -\bar \lambda^*) \rangle \\
 & =    \langle p^k- p^*, -B^*(\bar \lambda^* -\bar \lambda^{**}) \rangle + \langle p^*-p^{**}, -B^*(\bar \lambda^{k} -\bar \lambda^{**}) \rangle.
\end{aligned}
\end{equation}
Substituting \eqref{eq:inequality:relax:lambda:rho2} into \eqref{eq:relax:limit:unique:primal}, taking $x^{k_i}$ and $x^{k_j}$ instead of $x^{k}$ separately in the left hand side of \eqref{eq:relax:limit:unique:primal}, and remembering the right hand side of \eqref{eq:relax:limit:unique:primal} only has a unique limit according to \eqref{eq:inequality:relax:rho2}, together with notation $C^* = (2-\rho^*)/\rho^*$, we have
\begin{equation}
\begin{aligned}
&0 = \frac{r}{\rho^*} \|B(p^*-p^{**})\|^2 + \frac{1}{r \rho^*} \|\bar \lambda^* -\bar \lambda^{**}\|^2 + \frac{2(\rho^*-1)}{\rho^*}
\langle p^* -  p^{**}, B^*(\bar \lambda^* -\bar \lambda^{**}) \rangle  \\
&=\frac{\rho^*-1}{\rho^*}\|\sqrt{r}B(p^*-p^{**}) + \frac{1}{\sqrt{r}}(\bar \lambda^* - \bar \lambda^{**})\|^2 + rC^*\|B(p^*-p^{**})\|^2  + C^*\frac{1}{r}\|\bar \lambda^*-\bar \lambda^{**}\|^2. \notag
\end{aligned}
\end{equation}
Since $C^* >0$ while $\rho^* \in [1,2)$, what follows are $Bp^* = Bp^{**}$ and $\bar \lambda^*=\bar \lambda^{**}$.
The remaining proof is similar to the case for $\rho_k \in (0,1)$. Thus we get the uniqueness of the weak limits of all weakly convergent subsequences, which leads to the weak convergence of the iteration sequence $x^k$ produced by \eqref{admm:relax:sys:final} while $\rho_{k} \in[1,2)$.
\end{proof}
With these preparations, we could give the following ergodic convergence rate.
\begin{theorem}
  \label{thm:douglas-rachford-ergodic-rate-relax}
Supposing $x^*=(u^*, p^*, \bar \lambda^*)$ is a saddle-point of \eqref{eq:saddle-point-prob},
with assumption \eqref{eq:assumption:admm},
  then the following ergodic sequences with $(u^k, p^k, \bar \lambda^k)$ produced by iterations \eqref{admm:relax:sys:final},
  \begin{equation}
    u_{\erg}^k = \frac1k \sum_{k' = 1}^{k} u^{k'},\ \ \  p_{\erg}^k = \frac1k \sum_{k' = 1}^{k} p^{k'}, \quad
    \bar \lambda_{\erg}^k = \frac1k \sum_{k' = 1}^{k} \bar \lambda^{k'},
  \end{equation}
  converge weakly to $u^*$, $p^*$ and $\bar  \lambda^*$, respectively.
  Besides, the ``partial" primal-dual gap obeys $\gap_{x^*}(x_{\erg}) \geq 0$ where $x_{\erg}^k : = (u_{\erg}^k, p_{\erg}^k, \bar \lambda_{\erg}^k)$ and $\gap_{x^*}(x_{\erg})  = \mO(\frac{1}{k})$. In detail, for $0 < \rho_k < 1$,
  \begin{equation}
 \gap_{x^*}(x_{\erg}) \leq \frac 1k
    \
    \Bigl[
   \frac{r\norm{ B(p^{0} - p^*)}^2}{2}
    + \frac{\norm{\bar \lambda^{0} - \lambda^*}^2}{2r}+ \frac{1-\rho_{0}}{2\rho_{0}r} \|v^{0} -v^*\|^2  - \frac{2-\rho_{0}}{\rho_{0}}\langle p^{0} - p^{1}, -B^*(\bar \lambda^{0} -\bar \lambda^{1})\rangle
    \Bigr], \notag
  \end{equation}
where $v^* =  -rBp^* + \bar \lambda^*$, and for $\rho_k  \in [1,2)$, we have
    \begin{align}
\gap_{x^* }(x_{\erg}^k) &\leq \frac {1}{k\rho_0}
     \
    \Bigl[
 \frac{(2-\rho_0)r}{2}\|B(p^{0}-p^*)\|^2 +  \frac{2-\rho_0}{2r}\|\bar \lambda^{0} -\bar \lambda^* \|^2 \\
     &  + \frac{\rho_0-1}{2}\|\sqrt{r}B(p^0-p) + (\bar \lambda^0 - \bar \lambda)/\sqrt{r} \|^2- (2-\rho_{0})\langle p^{0} - p^{1}, -B^*( \lambda^{0} -\bar \lambda^{1})\rangle
    \Bigr]= \mO(\frac{1}{k}). \notag
  \end{align}
\end{theorem}
\begin{proof}
  First of all, by Theorem \ref{thm:gap:estimate:pre:u:weak}, $(u^k,p^k, \bar \lambda^k)$ converges
  weakly to a solution $(u^*,p^*, \bar \lambda^*)$
  of~\eqref{eq:saddle-point-prob} and  $(u^k, p^{k},\bar \lambda^k)$ converges weakly to a fixed point of \eqref{admm:original:sys:final}.
  To obtain the weak convergence of $\seq{u_{\erg}^k}$, we test with
  an $u$ and utilize the Stolz--Ces\`aro theorem to obtain
  \[
  \lim_{k \to \infty}  \scp{u_{\erg}^k}{x} = \lim_{k \to \infty}
  \frac{\sum_{k'=1}^k \scp{u^{k'}}{x}}{\sum_{k' = 1}^k 1} =
  \lim_{k \to \infty} \scp{u^k}{x} = \scp{u^*}{x}.
  \]
  The property $\wlim \bar \lambda_{\erg}^k = \bar \lambda^*$,  $\wlim  p_{\erg}^k = p^*$ can be proven analogously.

  Finally, in order to show the estimate on
  $\gap_{x^*}(u^k_{\erg}, p^k_{\erg}, \bar  \lambda^k_{\erg})$, we take $\rho_k \in [1,2)$ for example. Observing that the function
  $(u',p', \bar \lambda') \mapsto \mL(x',p',\bar \lambda^*) - \mL(x^*,y^*, \bar \lambda')$ is convex, together with the definition of saddle-points of \eqref{eq:saddle-point-prob}, \eqref{eq:estimate:relax:rho:ine:inclu:g}, \eqref{ieq:relax:rho:2}, \eqref{eq:inequality:relax:rho2:trans} and \eqref{eq:inequality:relax:rho2}, we have
  \begin{align*}
    &0 \leq \mL(u^k_{\erg}, p^k_{\erg}, \bar  \lambda^*) - \mL(u^*, p^*, \bar  \lambda^k_{\erg})
    \leq \frac1k   \sum_{k'=0}^{k-1} \mL(u^{k'+1}, p^{k'+1}, \bar  \lambda^*) - \mL(u^*, p^*, \bar  \lambda^{k'+1}) \\
    &\leq \frac1k   \Bigl\{\sum_{k' = 0}^{k-1} d_k(x^{k},x) - d_{k+1}(x^{k+1},x)
      - \sum_{k'=0}^{k-1} \frac{2-\rho_{k}}{\rho_{k}}\langle p^{k} - p^{k+1}, -B^*(\bar \lambda^{k} - \bar \lambda^{k+1}) \rangle   \Bigr\} \\
   & \leq \frac1k
      \Bigl\{  d_0(x^{0},x) - \frac{2-\rho_{0}}{\rho_{0}}\langle p^{0} - p^{1}, -B^*(\bar \lambda^{0} - \bar \lambda^{1}) \rangle \Bigr\}.
  \end{align*}
For $\rho_k \in (0,1)$, by \eqref{eq:gap:estimate:relaxed:rho1:origin} and \eqref{eq:estimate:relax:rho:ine:inclu:g}, we also get the corresponding estimate.
\end{proof}
Actually, if linear operators $A$, $B$ and the function $G$ have similar multiple separable structures, we can assign each separable component with different relaxation parameter as stated in the following corollary, which is more flexible in applications.
\begin{corollary}\label{cor:overrelaxation:multiple:parameters}
 From Theorem \ref{thm:gap:estimate:pre:u} and \ref{thm:gap:estimate:pre:u:weak}, noticing that if both $A$ and $B$ have multiple components and $G$ has multiple separable components, i.e.,
\[
  A=(A_{1}, A_{2}, \cdots, A_{L})^{T}, \quad B=(B_{1}, B_{2}, \cdots, B_{L})^{T},  \quad A_{i}u + B_{i}p_{i} = c_{i}, \quad i = 1,2, \cdots, L,
\]
 where $c = (c_1, \cdots, c_{L})^{T}$ and $G(p_1, \cdots,p_L) = \sum_{i=1}^{L}G_{i}(p_{i})$ as in \eqref{eq:saddle-point-prob},
the relaxed ADMM becomes
 \begin{equation} \label{admm:relax:sys:final:multi:R}
    \left\{
    \begin{aligned}
       u^{k+1}
      &= (rA^*A + \partial F)^{-1}[A^*(-rBp^k +rc - \bar{\lambda}^k) ] = (rA^*A + \partial F)^{-1}[\sum_{i=1}^L A_{i}^*(-rB_{i}p_{i}^k + rc_{i} -\bar \lambda_{i}^k)], \\
p_{1}^{k+1}  &= (r B_{1}^*B_{1} + \partial G_{1})^{-1}[B_{1}^*( -r \rho_{k}^{1} A_{1}u^{k+1} + r(1-\rho_{k}^{1})B_1p_{1}^{k} - \bar{\lambda}_{1}^{k} +r\rho_{k}^{1}c_{1})],
      \\
      &\cdots \\
     p_{L}^{k+1}  &= (r B_{L}^*B_{L} + \partial G_{L})^{-1}[B_{L}^*( -r \rho_{k}^{L} A_{L}u^{k+1} + r(1-\rho_{k}^{L})B_Lp_{L}^{k} - \bar{\lambda}_{L}^{k} +r\rho_{k}^{L}c_{L})],  \\
      \bar{\lambda}_{1}^{k+1} &= \bar{\lambda}_{1}^{k} + r(\rho_{k}^{1} A_{1}u^{k+1} + B_{1}p_{1}^{k+1}-(1-\rho_{k}^{1})B_{1}p_{1}^{k}-\rho_{k}^{1}c_{1}), \\
    &  \cdots \\
  \bar{\lambda}_{L}^{k+1} &= \bar{\lambda}_{L}^{k} + r(\rho_{k}^{L} A_{L}u^{k+1} + B_{L}p_{L}^{k+1}-(1-\rho_{k}^{L})B_{L}p_{L}^{k}-\rho_{k}^{L}c_{L}),
    \end{aligned}
    \right.
  \end{equation}
 by giving each component of $ p_{i}$ or $\bar \lambda_{i}$ relaxation parameter $\rho_{k}^{i}$ satisfying \eqref{eq:assum:relaxation}. Let us define
 \[
  \bar \lambda_{o,i}^{k+1} := (\bar \lambda_{i}^{k+1} - \bar \lambda_{i}^{k})/{\rho_{k}^i} + \bar \lambda_{i}^{k}, \quad  p_{o,i}^{k+1} := (p_{i}^{k+1} - p_{i}^{k})/\rho_{k}^i + p_{i}^{k}.
 \]
and denote $v_{i}^k$, $w_{i}^k$, $v_{o,i}^k$, $w_{o,i}^k$, $v_{i}$ and $w_{i}$ as each component of $v^k$, $w^k$, $v_{o}^k$, $w_{o}^k$, $v$ and $w$ respectively as defined in \eqref{def:v:w}, $i = 1,2, \cdots, L$. Assuming $A_iu_i + B_ip_i = c_i$, for $k \geq 0$, while $0< \rho_{k}^{l} <1$, for iteration \eqref{admm:relax:sys:final:multi:R}, we have
\begin{equation}%
   \begin{aligned}
  \gap_{x}(x^{k+1})  & \leq
\sum_{l=1}^{L}\bigl \{\rho_{k}^{l} r \langle B_{l}(p_{l}^{k} - p_{o,l}^{k+1}), B_{l}(p_{o,l}^{k+1} - p_{l}) \rangle +  \frac{\rho_{k}^{l}}{r}\langle \bar \lambda_{l}^{k} -\bar \lambda_{o,l}^{k+1}, \bar \lambda_{o,l}^{k+1} - \bar \lambda_{l} \rangle  \\
       &+ \frac{1-\rho_{k}^{l}}{r}\langle v_{l}^{k} -v_{o,l}^{k+1}, v_{o,l}^{k+1} -v_{l} \rangle + \rho_{k}^{l}(2-\rho_{k}^{l})\langle B_{l}(p_{l}^{k} - p_{o,l}^{k+1}), \bar \lambda_{l}^{k} -\bar \lambda_{o,l}^{k+1}\rangle \bigr\},
  \end{aligned}
\end{equation}
and while $1 \leq \rho_{k}^{l} <2$, we have
\begin{equation}
   \begin{aligned}
  \gap_{x }(x^{k+1})  & \leq \sum_{l=1}^{L}\bigl \{
(2-\rho_{k}^{l}) r \langle B_{l}(p_{l}^{k} - p_{o,l}^{k+1}), B_{l}(p_{o,l}^{k+1} - p_{l}) \rangle +  \frac{2-\rho_{k}^{l}}{r}\langle \bar \lambda_{l}^{k} -\bar \lambda_{o,l}^{k+1}, \bar \lambda_{o,l}^{k+1} - \bar \lambda_{l} \rangle \\
       &+\frac{\rho_{k}^{l}-1}{r} \langle w_{l}^{k} -w_{o,l}^{k+1}, w_{o,l}^{k+1} -w_{l} \rangle + \rho_{k}^{l}(2-\rho_{k}^{l})\langle B_{l}(p_{l}^{k} - p_{o,l}^{k+1}), \bar \lambda_{l}^{k} -\bar \lambda_{o,l}^{k+1}\rangle \bigr \}. \notag
  \end{aligned}
 \end{equation}
With these estimats, it could be checked that the results of Theorem  \ref{thm:gap:estimate:pre:u:weak} and \ref{thm:douglas-rachford-ergodic-rate-relax} remain valid with minor adjustments.
\end{corollary}
\section{Relaxed ADMM and Douglas-Rachford splitting method}\label{sec:pre:admm:relax}
\subsection{Preconditioned ADMM with relaxation}
Let $p=(p_{1}, p_{2})^{T} \in \bar{Y} = Y \times Y_{2}$, $\tilde A = (A, S)^{T}$ with $\tilde A u = (Au, Su)^{T}$, $\tilde{B}p = (Bp_{1}, -p_{2})^{T}$, $\tilde c = (c,0)^{T}$, and $\tilde{G}(p) = G(p_{1}) +I_{\{0\}}^*(p_2) $ with $I_{\{0\}}^*$ being the conjugate of the indicator function $I_{\{0\}}$. Then the original model \eqref{saddle:primal:original} could be reformulated as
 \begin{equation}\label{saddle:primal}
 \min_{u \in X, p \in \bar Y} F(u) + \tilde{G}(p), \quad \text{subject to} \quad \tilde Au +\tilde{B} p = \tilde c,
 \end{equation}
 where $u$ is a block and $p$ is another block.
 We first write the augmented Lagrangian,
 \begin{equation}\label{augmented_lagrangian:modified}
  \tilde L_{r}(u, p, \bar \lambda): = F(u) + \tilde{G}(p) + \langle \bar{\lambda}, \tilde Au+ \tilde Bp- \tilde c\rangle + \frac{r}{2}\|\tilde Au + \tilde B p- \tilde c \|^2,
\end{equation}
where $\bar{\lambda} = (\bar{\lambda}_{1} , \bar{\lambda}_{2})^{T} \in \bar Z  =Z \times Z_2$.
Actually, by the definition of the Fenchel conjugate \cite{HBPL,KK},  \eqref{saddle:primal} is equivalent to the following saddle-point problem and the dual problem,
 \begin{equation}
  \label{eq:precon-saddle-point:M}
\min_{u \in X, p \in \bar Y} \max_{\bar \lambda \in \bar Z} F(u) + \tilde G(p)+ \langle \bar \lambda, \tilde Au + \tilde B p- \tilde c \rangle  \quad \Leftrightarrow \quad  \min_{\bar \lambda\in \bar Z} F^*(-\tilde A^* \bar \lambda)+ \tilde G^*(-\tilde B^* \bar \lambda) + \langle \tilde  c,  \bar \lambda \rangle,
\end{equation}
where $\tilde{G}^*(\bar \lambda) = G^*(\bar \lambda_{1}) + I_{\{0\}}(\bar \lambda_{2})$. \eqref{eq:precon-saddle-point:M} has a saddle-point solution $(u^*, p^*, \bar \lambda_1^*,0)$ with $(u^*,p_{1}^*, \bar \lambda_{1}^*)$ is a saddle-point solution of the original saddle-point problem \eqref{eq:saddle-point-prob}, through checking the optimality conditions and constraint of both two problems in \eqref{eq:precon-saddle-point:M}.
\begin{proposition}
In the case of $M=rB^*B$, \eqref{admm:relax:sys:full:final} could be recovered by applying Corollary \ref{cor:overrelaxation:multiple:parameters} to the modified constrained problem \eqref{saddle:primal}.
\end{proposition}
\begin{proof}
Applying Corollary \eqref{cor:overrelaxation:multiple:parameters} to the modified constrained problem \eqref{saddle:primal}, we have
 \begin{equation}
    \label{pre:admm:sys:referee:relax}
    \left\{
    \begin{aligned}
      u^{k+1}
      &= (rA^*A + r S^*S + \partial F)^{-1}[A^*(-rBp_{1}^k+rc - \bar{\lambda}_{1}^k) + rS^*Su^{k}] ,
      \\
      p_{1}^{k+1}  &= (r B^*B + \partial G)^{-1}[B^*(-\bar{\lambda}_{1}^{k}  -r \rho_{k}^{1}A u^{k+1}+ r(1-\rho_{k}^{1})Bp_{1}^{k}+\rho_{k}^{1}c)],
      \\
      p_{2}^{k+1} & = (r I + \partial I_{\{0\}}^* )^{-1}[\bar{\lambda}_{2}^{k} + r (\rho_{k}^{2}S u^{k+1}+(1-\rho_{k}^{2})p_{2}^{k})], \\
      \bar{\lambda}_{1}^{k+1} &= \bar{\lambda}_{1}^{k} + r[\rho_{k}^{1}A u^{k+1}- (1-\rho_{k}^{1})Bp_{1}^{k} +B p_{1}^{k+1}-\rho_{k}^{1}c], \\
      \bar{\lambda}_{2}^{k+1} &= \bar{\lambda}_{2}^{k} + r[\rho_{k}^{2}S u^{k+1}+(1-\rho_{k}^{2})p_{2}^{k} - p_{2}^{k+1}].
    \end{aligned}
    \right.
  \end{equation}
By the Moreau identity \cite{CP}, for the update of $p_2^{k+1}$ in \eqref{pre:admm:sys:referee:relax}, we have
\begin{equation}\label{eq:p:M:modi:UP}
rp_{2}^{k+1} = \bar{\lambda}_{2}^{k} + r (\rho_{k}^{2}S u^{k+1}+(1-\rho_{k}^{2})p_{2}^{k}).
\end{equation}
 Together with the update of $\bar{\lambda}_{2}^{k+1}$, we have
\begin{equation}\label{eq:second:lambda:relax}
\bar{\lambda}_{2}^{k+1} \equiv 0.
\end{equation}  If setting $\rho_{k}^{2} \equiv 1$, for the update of $p_{2}^{k+1}$, we have
  \[
  rp_{2}^{k+1} = \bar{\lambda}_{2}^{k} + r S u^{k+1}.
  \]
What follows is
\begin{equation}\label{eq:second:p:relax}
  p_{2}^{k+1} =  S u^{k+1}\ \ \Rightarrow \ \  p_{2}^{k} =  S u^{k}.
\end{equation}
Substituting \eqref{eq:second:p:relax} into the update of $u^{k+1}$ in \eqref{pre:admm:sys:referee:relax}, we get the relaxed and preconditioned ADMM as follows, with notation $N = rA^*A + rS^*S$:
  \begin{equation}
    \label{pre:admm:sys:referee:relax:final} \tag{rpADMM}
    \left\{
    \begin{aligned}
 u^{k+1}
      &= (N + \partial F)^{-1}[A^*(-rBp_{1}^k+rc - \bar{\lambda}_{1}^k) + (N-rA^{*}A)u^{k}],
      \\
     p_{1}^{k+1}  &= (r B^*B + \partial G)^{-1}[B^*(-\bar{\lambda}_{1}^{k}  -r \rho_{k}^{1}A u^{k+1}+ r(1-\rho_{k}^{1})Bp_{1}^{k}+\rho_{k}^{1}c)],
      \\
\bar{\lambda}_{1}^{k+1} &= \bar{\lambda}_{1}^{k} + r[\rho_{k}^{1}A u^{k+1}- (1-\rho_{k}^{1})Bp_{1}^{k} +B p_{1}^{k+1}-\rho_{k}^{1}c].
    \end{aligned}
    \right.
  \end{equation}
By choosing $\rho_{k}^{1} \equiv 1$, we get the preconditioned ADMM,
   \begin{equation}
    \label{pre:admm:sys:final}\tag{pADMM}
    \left\{
    \begin{aligned}
       u^{k+1}
      &= (rA^*A + r S^*S + \partial F)^{-1}[A^*(-rBp_{1}^k +rc - \bar{\lambda}_{1}^k) + (N - rA^*A) u^k),
      \\
p_{1}^{k+1}  &= (r B^*B + \partial G)^{-1}[B^*( -rA^*u^{k+1} - \bar{\lambda}_{1}^{k} +rc)],
      \\
      \bar{\lambda}_{1}^{k+1} &= \bar{\lambda}_{1}^{k} + r(Au^{k+1} +B p_{1}^{k+1}-c).
    \end{aligned}
    \right.
  \end{equation}
\end{proof}
Here, we discuss some connections to the existing works.
\begin{remark}
It is not new to obtain the preconditioned ADMM or other proximal ADMM with or without relaxation by applying the classical ADMM or relaxed ADMM in \cite{EP} to the modified constrained problems including \eqref{saddle:primal}. For example, similar idea could be found in \cite{BS3} for designing preconditioned version of accelerated Douglas-Rachford splitting method. In finite dimensional spaces, a similar modified constraint problem is also considered in \cite{NJ} for linearization of proximal ADMM, which is actually $\tilde{G}(p) = G(p_{1})$ without the constraint $I_{\{0\}}^*(p_2)$ in \eqref{saddle:primal}. However, \cite{NJ} direct applied relaxed ADMM in \cite{EP} with only one relaxation parameter, i.e., $\rho_k^1 = \rho_k^2   \equiv \rho_0 \in(0,2)$. Thus the update of $p_{2}^{k+1}$ in \eqref{eq:p:M:modi:UP} could not be dropped and cooperated into the updates of $u^{k+1}$ in \cite{NJ}, and there are 4 updates for the linearized ADMM in \cite{NJ} instead of 3 updates as in \eqref{pre:admm:sys:referee:relax:final}. Hence, Corollary \ref{cor:overrelaxation:multiple:parameters} could bring out compact and probably more efficient iterations.

\end{remark}
\subsection{The corresponding Douglas-Rachford splitting}\label{sec:admm-equavi-dr}
An interesting relation between the Douglas-Rachford splitting method and ADMM was first pointed out by Gabay \cite{DGBA}, which was also further studied in \cite{SET, Tai}. Here we extend it to relaxation variants with additional discussion.
 The optimality condition for the dual problem \eqref{eq:precon-saddle-point:M} is
\begin{equation}\label{eq:dual}
0 \in r\partial ( F^*(- \tilde A^* \lambda)) +r\partial (\tilde{G}^*(-\tilde B^* \lambda)) + r \tilde c, \quad \text{for any} \quad r >0.
\end{equation}
Then we can apply Douglas-Rachford splitting method to the operator splitting by $r\partial ( F^*(-\tilde A^* \cdot))$ and $r\partial \tilde{G}^*(-\tilde B^* \cdot) + r \tilde c$. As suggested in \cite{BS}, we can write the Douglas-Rachford as
\begin{equation}\label{iter:proximal}
 0 \in \mM (x^{k+1} -x^{k})+ \mA x^{k+1}
\end{equation}
where $x^{k}= (\lambda^{k}, \tilde{v}^{k})^{T}$ and
\[
\mM = \begin{pmatrix}
I &-I \\
-I & I
\end{pmatrix}, \quad
\mA = \begin{pmatrix}
r\partial \tilde{G}^*(-\tilde B^* \cdot) + r \tilde c &I \\
-I & \{r\partial ( F^*(- \tilde A^* \cdot))\} ^{-1}
\end{pmatrix}.
\]
Writting \eqref{iter:proximal} component wisely, we have
\begin{equation}
\begin{cases}
\lambda^{k+1} = J_{r A_{\tilde B}}(\lambda^{k} - \tilde{v}^{k}), \\
\tilde{v}^{k+1}=  J_{(r A_{\tilde A})^{-1}}(2 \lambda^{k+1}-(\lambda^{k} - \tilde{v}^{k})),
\end{cases}
\end{equation}
where $J_{R}= (I +R)^{-1}$ is the resolvent of the maximal monotone operator $R$ and $rA_{\tilde B} = r \partial (\tilde{G}^*(-\tilde B^* \cdot)) + r \tilde c$, $A_{\tilde A} =  \partial ( F^*(-\tilde A^* \cdot))$. Denoting $v^{k}= \lambda^{k} - \tilde{v}^{k}$ and using the Moreau's identity, we
have
\begin{equation}\label{pre:right}
\begin{cases}
\lambda^{k+1} = J_{r A_{\tilde B}}(v^{k}), \tag{DR}\\
v^{k+1}=  v^{k} + J_{r A_{\tilde A}}(2 \lambda^{k+1}-v^{k}) - \lambda^{k+1}.
\end{cases}
\end{equation}
Furthermore, adding relaxation to \eqref{pre:right}, we have
\begin{equation}\label{pre:right:rela}
\begin{cases}
\lambda^{k+1} = J_{r A_{\tilde B}}(v^{k}), \tag{rDR}\\
v^{k+1}=  v^{k} + I_{\rho}[J_{r A_{\tilde A}}(2 \lambda^{k+1}-v^{k}) - \lambda^{k+1}],
\end{cases}
\end{equation}
where
\[
I_{\rho,k} = \text{Diag} \ [\rho_{k}^{1}, \rho_{k}^{2}], \quad \text{with} \quad \rho_{k}^{i} \ \ \text{satisfying \eqref{eq:assum:relaxation}}, \ \ i = 1,2.
\]
Originated from Gabay \cite{DGBA}, we can get the following theorem.
\begin{theorem}\label{thm:equivalent:dr:admm:relax}
With the assumption \eqref{eq:assumption:admm}, preconditioned ADMM \eqref{pre:admm:sys:final} and its relaxation \eqref{pre:admm:sys:referee:relax:final} could be recovered from \eqref{pre:right} and \eqref{pre:right:rela} separately.
\end{theorem}
\begin{proof}
For $J_{r A_{\tilde A}}$ and $J_{r A_{\tilde B}}$, it could be checked that by Proposition 4.1 of \cite{DGBA} (or similar arguments as in Proposition 23.23 of \cite{HBPL}), while $(r\tilde A^*\tilde A+ \partial F)^{-1}$ and $(rB^*B+\partial G)^{-1}$ exist and are Lipschitz continuous which could be guaranteed by \eqref{eq:assumption:admm}, they have the following forms
\begin{align}
J_{r A_{\tilde A }} &= (I  + (-r\tilde A \circ \partial F^* \circ -\tilde A^*))^{-1} = I + r \tilde A \circ(r \tilde A^* \tilde A+ (\partial F^*)^{-1})^{-1}\circ -\tilde A^*, \\
J_{r A_{\tilde B }} & = (\{I + r  B \circ(r B^*  B+ \partial  G)^{-1}\circ - B^*\}(\cdot - r  c), \ \  \{I -r (r I+ \partial I_{\{0\}}^*)^{-1} \}(\cdot))^{T},
\end{align}
where the second component of $J_{r A_{\tilde B }} $ can be simplified with the Moreau identity \cite{CP},
\[
\{I -r (r I+ \partial I_{\{0\}}^*)^{-1} \}(\cdot) = J_{r \partial I_{\{0\}}}(\cdot) = (I + r \partial I_{\{0\}})^{-1}(\cdot).
\]
Together with $(\partial F^*)^{-1} = \partial F$ and denoting $\bar{\lambda}^{k} := \lambda^{k+1}$,
let us introduce some variables
\begin{align}\label{p:intro}
u^{k+1}: &= (r A^*A + rS^*S + \partial F)^{-1}( \tilde A^*(-2 \bar \lambda^{k} + v^{k})), \\
p_{1}^{k} :& = (rB^*B + \partial G)^{-1}(-B^*(v_{1}^{k}-rc)), \label{p1:intro}\\
p_{2}^{k} :& = (rI + \partial I_{\{0\}}^*)^{-1}(v^{k}). \label{p2:intro}
\end{align}
 Substituting the above relations to \eqref{pre:right:rela}, we have
\begin{equation}\label{pre:fine:compo}
\begin{cases}
\bar{\lambda}_{1}^{k} = v_{1}^{k} - rc + rBp_{1}^{k}, \\
\bar{\lambda}_{2}^{k} = v_{2}^{k} - rp_{2}^{k}  = J_{r \partial I_{\{0\}}}(v_{2}^{k}) = 0 , \\
v_{1}^{k+1}=  (I - \rho_{k}^{1})v_{1}^{k} + \rho_{k}^{1} \bar{\lambda}_{1}^{k} + r \rho_{k}^{1}A u^{k+1}, \\
v_{2}^{k+1}=  (I - \rho_{k}^{2})v_{2}^{k} + \rho_{k}^{2} \bar{\lambda}_{2}^{k} + r \rho_{k}^{2}S u^{k+1}.
\end{cases}
\end{equation}
 By
\begin{equation}\label{eq:update:v1:middle}
\bar{\lambda}^{k} = v^{k} - r \tilde c + r\tilde B p^{k},
\end{equation}
we have
\[
-2 \bar{\lambda}^{k}+v^{k} = -r\tilde Bp^{k} -  \bar{\lambda}^{k} + r \tilde c,
\]
where
\begin{equation}\label{update:u:pri}
u^{k+1}=(r A^*A + rS^*S +  \partial F)^{-1}[A^*(-rBp_{1}^{k} -  \bar{\lambda}_{1}^{k}+rc)+S^*(rp_{2}^{k} -  \bar{\lambda}_{2}^{k})].
\end{equation}
By the second equation of \eqref{pre:fine:compo}, we have
\[
\bar{\lambda}_{2}^{k} \equiv 0, \quad \text{and} \quad
v_{2}^{k} \equiv r p_{2}^{k}.
\]
If we set $\rho_{k}^{2} \equiv 1$,
by the update of $v_{2}^{k+1}$ in \eqref{pre:fine:compo}, we have
\[
v_{2}^{k+1}=   r S u^{k+1},
\]
and what follows is
\[
r p_{2}^{k}=  r S u^{k}.
\]
Substituting it in to \eqref{update:u:pri}, we have the final form for updating $u^{k+1}$,
\begin{equation}\label{update:u:pri:2}
u^{k+1}=(r A^*A + rS^*S +  \partial F)^{-1}[A^*(-rBp_{1}^{k} -  \bar{\lambda}_{1}^{k}+rc)+ r S^*Su^{k}].
\end{equation}
If we set $N = r A^*A+ rS^*S$, \eqref{update:u:pri} becomes
\begin{equation*}
u^{k+1}=(N +  \partial F)^{-1}[A^*(-rBp_{1}^{k} -  \bar{\lambda}_{1}^{k}+rc)+ (N-rA^*A)u^{k}].
\end{equation*}
This is exactly our preconditioned ADMM with only preconditioning for updating $u^{k+1}$.
We have figured clear the updates of $u^{k+1}$, $\bar{\lambda}_{2}^{k+1}$, $p_{2}^{k+1}$, $v_{2}^{k+1}$. Let us consider the updates of $\bar{\lambda}_{1}^{k+1}$, $v_{1}^{k+1}$, $p_{1}^{k+1}$. Starting from \eqref{pre:fine:compo}, by the update of $v_{1}^{k+1}$, we have
\begin{align}
\bar{\lambda}_{1}^{k+1} &= v_{1}^{k+1} + rB p_{1}^{k+1}-rc \notag \\
& = (I - \rho_{k}^{1})v_{1}^{k} + \rho_{k}^{1} \bar{\lambda}_{1}^{k} + r \rho_{k}^{1}A u^{k+1} + rB p_{1}^{k+1}-rc  \notag \\
&= \bar{\lambda}_{1}^{k} + [(1 - \rho_{k}^{1})(v_{1}^{k}  -\bar{\lambda}_{1}^{k} ) +r \rho_{k}^{1}A u^{k+1} + rB p_{1}^{k+1}-rc] \notag \\
& =  \bar{\lambda}_{1}^{k}  + r[(1 - \rho_{k}^{1})(-Bp_{1}^{k}+c) +  \rho_{k}^{1}A u^{k+1} + B p_{1}^{k+1}-c] \notag \\
& =  \bar{\lambda}_{1}^{k}  + r[\rho_{k}^{1}A u^{k+1} + B p_{1}^{k+1}-(1 - \rho_{k}^{1})Bp_{1}^{k}  - \rho_{k}^{1}c]. \label{eq:lambda:update:fin}
\end{align}
For the update of $p_{1}^{k+1}$, by \eqref{p1:intro} for $k+1$,
substituting the update of $\bar{\lambda}_{1}^{k+1} $ into the following equation
\begin{equation}\label{eq:relation:p:lambda}
 -B^* \bar{\lambda}_{1}^{k+1} \in \partial G(p_{1}^{k+1}),
\end{equation}
with \eqref{eq:lambda:update:fin}, we have
\[
 -B^*\{\bar{\lambda}_{1}^{k}  + r[\rho_{k}^{1}A u^{k+1} -(1 - \rho_{k}^{1})Bp_{1}^{k}  - \rho_{k}^{1}c]\} \in  \partial G(p_{1}^{k+1}) + r B^*B p_{1}^{k+1}.
\]
What follows are the update of $p_{1}^{k+1}$
\[
p_{1}^{k+1} =  (rB^*B + \partial G)^{-1}( -B^*\{\bar{\lambda}_{1}^{k}  + r[\rho_{k}^{1}A u^{k+1} -(1 - \rho_{k}^{1})Bp_{1}^{k}  - \rho_{k}^{1}c]\}),
\]
and the update of $v_{1}^{k+1}$
\[
 v_{1}^{k+1} = \bar{\lambda}_{1}^{k+1} - r Bp_{1}^{k+1} +rc.
\]
Collecting the updates of $u^{k+1}$, $p_{1}^{k+1}$ and $\bar{\lambda}_{1}^{k+1}$, we get the relaxed and preconditioned ADMM from relaxed Douglas-Rachford splitting, i.e., \eqref{pre:admm:sys:referee:relax:final}, with $N = r A^*A + rS^*S$.
Thus we get relaxed and preconditioned ADMM \eqref{pre:admm:sys:referee:relax:final} from the corresponding Douglas-Rachford methods \eqref{pre:right:rela}. By setting $\rho_{k}^{1} \equiv 1$, we also recover \eqref{pre:admm:sys:final} from \eqref{pre:right}.

\end{proof}

\section{Fully preconditioned Douglas-Rachford splitting method and its relaxation}\label{sec:full:pre:admm}

In this section, we will first give the ``partial" gap estimates.
\begin{theorem}\label{them:gap:estimate:full:pre:lao}
For the iteration \eqref{admm:relax:sys:full:final}, with the same notation $p_o^{k}$, $\bar \lambda_o^{k}$ as in \eqref{eq:def:po:lambdao}, $v$, $w$ as in \eqref{def:v:w} and the assumption $Au+Bp = c$, for the ``partial" primal-dual gap for the fully relaxed ADMM, denoting $ \gap_{ x}(x^{k+1}): =  \gap_{(u, p, \bar \lambda)}(u^{k+1},p^{k+1}, \bar \lambda^{k+1})$ with $x = (u,p, \bar \lambda)$ and $x^{k+1} = (u^{k+1}, p^{k+1}, \bar \lambda^{k+1})$, we have, for $k \geq 0$, while $0< \rho_{k} <1$,
\begin{equation} \label{eq:gap:estimate:relaxed:rho1:full}
   \begin{aligned}
\gap_{ x}(x^{k+1}) & \leq
\rho_{k} r \langle B(p^{k} - p_{o}^{k+1}), B(p_{o}^{k+1} - p) \rangle +  \frac{\rho_{k}}{r}\langle \bar \lambda^{k} -\bar \lambda_{o}^{k+1}, \bar \lambda_{o}^{k+1} - \bar \lambda \rangle  \\
       &+ \frac{1-\rho_{k}}{r}\langle v^{k} -v_{o}^{k+1}, v_{o}^{k+1} -v \rangle + \rho_{k}(2-\rho_{k})\langle B(p^{k} - p_{o}^{k+1}), \bar \lambda^{k} -\bar \lambda_{o}^{k+1}\rangle\\
       &+\langle (N-rA^*A)(u^k-u^{k+1}), u^{k+1}-u \rangle+ \langle (M-rB^*B)(p^k-p^{k+1}), p^{k+1}-p \rangle, \notag
  \end{aligned}
\end{equation}
and while $1 \leq \rho_{k} <2$,
\begin{equation}
   \begin{aligned}
  \gap_{ x}(x^{k+1})  & \leq
(2-\rho_{k}) r \langle p^{k} - p_{o}^{k+1}, p_{o}^{k+1} - p \rangle +  \frac{2-\rho_{k}}{r}\langle \bar \lambda^{k} -\bar \lambda_{o}^{k+1}, \bar \lambda_{o}^{k+1} - \bar \lambda \rangle  \label{eq:gap:estimate:relaxed:rho2:full} \\
       &+\frac{\rho_{k}-1}{r} \langle w^{k} -w_{o}^{k+1}, w_{o}^{k+1} -w \rangle + \rho_{k}(2-\rho_{k})\langle B(p^{k} - p_{o}^{k+1}), \bar \lambda^{k} -\bar \lambda_{o}^{k+1}\rangle \notag \\
       & +\langle (N-rA^*A)(u^k-u^{k+1}), u^{k+1}-u \rangle+ \langle (M-rB^*B)(p^k-p^{k+1}), p^{k+1}-p \rangle. \notag
  \end{aligned}
\end{equation}
\end{theorem}

\begin{proof}
The proof is quite similar to Theorem \ref{thm:gap:estimate:pre:u}. By direct calculation, the updates of $u^{k+1}$ and $p^{k+1}$ in \eqref{admm:relax:sys:full:final} could be reformulated as follows
\begin{align*}
&A^*(-rBp^k +rc - \bar{\lambda}^k) -rA^*Au^{k+1} +(N-rA^*A)(u^k-u^{k+1}) \in \partial F(u^{k+1}), \\
&B^*( -r \rho_{k} Au^{k+1} + r(1-\rho_{k})Bp_{1}^{k} - \bar{\lambda}_{1}^{k} +r\rho_{k}c) - rB^*Bp^{k+1} +(M-rB^*B)(p^k-p^{k+1}) \in \partial G(p^{k+1}).
\end{align*}
With the update of $\bar \lambda^{k+1}$ in \eqref{admm:relax:sys:full:final}, the second equation of above equations could be written as
\begin{equation}\label{eq:reformulate:g:inclusion}
-B^*\bar \lambda^{k+1} +  (M-rB^*B)(p^k-p^{k+1}) \in \partial G(p^{k+1}).
\end{equation}
  By the notation $p_o^{k}$, $\bar \lambda_o^{k}$ as in \eqref{eq:def:po:lambdao}, it could be seen the update of $\bar \lambda^{k+1}$  in \eqref{admm:relax:sys:full:final} becomes
\[
\bar \lambda_o^{k+1} = \bar \lambda^{k} + r(Au^{k+1}+Bp_o^{k+1}-c),
\]
and what follows is
\[
-rBp^{k} - rAu^{k+1} +rc -\bar \lambda^{k} =-\bar \lambda_{o}^{k+1} +rB p_{o}^{k+1} - rBp^{k}.
\]
Thus the update of $u^{k+1}$ could be written as
\begin{equation}\label{eq:reformulate:f:inclusion}
A^*(-\bar \lambda_{o}^{k+1} +rB p_{o}^{k+1} - rBp^{k}) + (N-rA^*A)(u^k-u^{k+1}) \in \partial F(u^{k+1}).
\end{equation}
With \eqref{eq:reformulate:g:inclusion} and \eqref{eq:reformulate:f:inclusion}, by similar proof as in Theorem \ref{thm:gap:estimate:pre:u} with additional terms involving $N-rA^*A$ and $M-rB^*B$, we see
\begin{align*}
\gap_{x}(x^{k+1})  &= \mL(u^{k+1}, p^{k+1}, \bar \lambda) -\mL(u, p, \bar \lambda^{k+1}) \leq \frac{1}{r}\langle  \bar \lambda^{k} -\bar \lambda_o^{k+1}, \bar \lambda_o^{k+1} - \lambda  \rangle \\
& + r\langle B(p^k - p_o^{k+1}), B(p_o^{k+1}-p) \rangle - \langle p_o^{k+1} - p^k, -B^*(\bar \lambda_o^{k+1}  -\bar \lambda^{k}) \rangle  \\
& -(1-\rho_{k}) \langle B(p^{k} - p_o^{k+1}),\bar \lambda_o^{k+1} -\lambda \rangle -(1-\rho_{k}) \langle B(p_o^{k+1} -p), \bar \lambda^{k} - \bar \lambda_o^{k+1} \rangle \\
& - (1-\rho_{k})^2\langle B(p^k -p_o^{k+1}), \bar \lambda^{k} - \bar \lambda_o^{k+1} \rangle \\
& + \langle (N-rA^*A)(u^k-u^{k+1}), u^{k+1}-u \rangle + \langle (M-rB^*B)(p^k-p^{k+1}), p^{k+1}-p \rangle,
\end{align*}
which is completely similar to \eqref{eq:ref:I}, \eqref{eq:II:final:esti} with the last two additional inner products involving $(N-rA^*A)$ and $(M-rB^*B)$. Thus, with \eqref{eq:gap:estimate:demtalil}, \eqref{eq:expan:v} and \eqref{eq:expan:w}, one gets this theorem with similar arguments.
\end{proof}

\begin{corollary}\label{cor:fully:relax:F}
For iteration \eqref{admm:relax:sys:full:final}, with $v$, $w$ in \eqref{def:v:w} and assumption $Au+Bp=c$, $N - rA^*A = rS^*S$ and $M -rB^*B = rH^*H$, for the primal-dual gap for the relaxed ADMM, we have, for $k\geq 1$, while $0< \rho_k <1$,
\begin{equation} 
   \begin{aligned}
  \gap_{ x}(x^{k+1}) &\leq
 \frac{r}{2}(\|B(p^{k}-p)\|^2-\|B(p^{k+1}-p)\|^2) +  \frac{1}{2r}(\|\bar \lambda^{k} -\bar \lambda \|^2 -\|\bar\lambda^{k+1} -\bar \lambda \|^2 ) \\
       &+ \frac{1-\rho_{k}}{2\rho_{k}r}(\|v^{k}-v\|^2-\|v^{k+1}-v\|^2)
       - \frac{r(2-\rho_{k})}{2\rho_{k}}\|B(p^{k}-p^{k+1})\|^2\\
    &-\frac{(2-\rho_{k})}{2r\rho_{k}}\|\bar\lambda^{k}-\bar\lambda^{k+1}\|^2
       - \frac{1-\rho_{k}}{r}\frac{2-\rho_{k}}{2\rho_{k}^2}\|v^{k}-v^{k+1}\|^2-\frac{r}{\rho_k}\|H(p^{k+1}-p^{k})\|^2\\
      &  +(\frac{r}{\rho_k} -\frac{r}{2})\|H(p^k-p^{k-1})\|^2 +\frac{r}{2}(\|H(p^{k}-p)\|^2 - \|H(p^{k+1}-p)\|^2) \\
       &+ \frac{r}{2}(\|S(u^{k}-u)\|^2 - \|S(u^{k+1} -u)\|^2 - \|S(u^{k+1}-u^k)\|^2). \notag
  \end{aligned}
\end{equation}
While $1 \leq \rho_k <2$, we have
\begin{equation} \label{eq:gap:estimate:relaxed:rho2:origin:full}
   \begin{aligned}
  \gap_{ x}(x^{k+1})  &\leq
 \frac{r}{2\rho_{k}}(\|B(p^{k}-p)\|^2-\|B(p^{k+1}-p)\|^2) +  \frac{1}{2r\rho_{k}}(\|\bar \lambda^{k} -\bar \lambda \|^2 -\|\bar\lambda^{k+1} -\bar \lambda \|^2 ) \\
& + \frac{\rho_{k}-1}{\rho_{k} }[-\langle p^k-p', -B^*(\bar \lambda^k -\bar \lambda') \rangle + \langle p^{k+1}-p', -B^*(\bar \lambda^{k+1} -\bar \lambda') \rangle]\\
       &-\frac{(2-\rho_{k})^2}{2r\rho_{k}^2}\|\bar\lambda^{k}-\bar\lambda^{k+1}\|^2
       - \frac{\rho_{k}-1}{r}\frac{2-\rho_{k}}{2\rho_{k}^2}\|w^{k}-w^{k+1}\|^2  -\frac{r}{\rho_k}\|H(p^{k+1}-p^{k})\|^2\\
      &  +(\frac{r}{\rho_k} -\frac{r}{2})\|H(p^k-p^{k-1})\|^2 +\frac{r}{2}(\|H(p^{k}-p)\|^2 - \|H(p^{k+1}-p)\|^2) \\
       &+ \frac{r}{2}(\|S(u^{k}-u)\|^2 - \|S(u^{k+1} -u)\|^2 - \|S(u^{k+1}-u^k)\|^2) - \frac{r(2-\rho_{k})^2}{2\rho_{k}^2}\|B(p^{k}-p^{k+1})\|^2. \notag
  \end{aligned}
\end{equation}
\end{corollary}
\begin{proof}
The proof is similar to Corollary \ref{cor:relax:partial}. Taking $0 < \rho_k <1 $ for example, after substituting $p_o^{k+1}$ and $\lambda_o^{k+1}$ into the gap estimate in Theorem \ref{them:gap:estimate:full:pre:lao} by the expansion of $p_o^{k+1}$ and $\lambda_o^{k+1}$ by as in \eqref{eq:def:po:lambdao}, we have the detailed ``partial" gap estimate,
 \begin{equation} \label{eq:gap:estimate:relaxed:rho1:origin:fullyE}
   \begin{aligned}
  &\gap_{x}(x^{k+1})   \leq
 \frac{r}{2}(\|B(p^{k}-p)\|^2-\|B(p^{k+1}-p)\|^2) +  \frac{1}{2r}(\|\bar \lambda^{k} -\bar \lambda \|^2 -\|\bar\lambda^{k+1} -\bar \lambda \|^2 ) \\
       &-\frac{(2-\rho_{k})}{2r\rho_{k}}\|\bar\lambda^{k}-\bar\lambda^{k+1}\|^2 + \frac{1-\rho_{k}}{2\rho_{k}r}(\|v^{k}-v\|^2-\|v^{k+1}-v\|^2) - \frac{1-\rho_{k}}{r}\frac{2-\rho_{k}}{2\rho_{k}^2}\|v^{k}-v^{k+1}\|^2\\
       &- \frac{2-\rho_{k}}{\rho_{k}}\langle -B(p^{k} - p^{k+1}), \bar \lambda^{k} -\bar \lambda^{k+1}\rangle + \langle (M-rB^*B)(p^k-p^{k+1}), p^{k+1}-p \rangle \\
       &+ \langle(N-rA^*A)(u^k-u^{k+1}),u^{k+1}-u \rangle  - \frac{r(2-\rho_{k})}{2\rho_{k}}\|B(p^{k}-p^{k+1})\|^2.
  \end{aligned}
\end{equation}
It could be verified that
\begin{align}
&\langle (M - rB^*B)(p^k-p^{k+1}),p^{k+1}-p \rangle- \langle p^{k+1} - p^{k}, -B^*(\bar \lambda^{k+1} -\bar \lambda^{k})\rangle \notag \\
&  = - \langle p^{k+1} - p^{k}, (M - rB^*B)(p^k-p^{k+1})-B^*\bar \lambda^{k+1} - [ (M - rB^*B)(p^{k-1}-p^{k})- B^*\bar \lambda^{k}])\rangle \notag \\
& + \langle p^{k+1}-p^{k}, (M - rB^*B)(p^{k}-p^{k+1}) \rangle - \langle p^{k+1}-p^{k}, (M-rB^*B)(p^{k-1}-p^{k}) \rangle \notag \\
& + \langle (M - rB^*B)(p^k-p^{k+1}),p^{k+1}-p \rangle. \notag
\end{align}
By  \eqref{eq:reformulate:g:inclusion} and monotone property of subgradient $\partial G$, we see for $k \geq 1$,
\[
- \langle p^{k+1} - p^{k}, (M - rB^*B)(p^k-p^{k+1})-B^*\bar \lambda^{k+1} - [ (M - rB^*B)(p^{k-1}-p^{k})- B^*\bar \lambda^{k}])\rangle  \leq 0,
\]
and with the assumption $M - rB^*B \geq 0$, we get
\begin{align}
&\langle (M - rB^*B)(p^k-p^{k+1}),p^{k+1}-p \rangle- \langle p^{k+1} - p^{k}, -B^*(\bar \lambda^{k+1} -\bar \lambda^{k})\rangle \notag \\
&\leq \langle (M - rB^*B)(p^{k+1} - p^{k}), 2(p^k-p^{k+1}) +p - p^{k-1} \rangle  \notag \\
& = -2  \langle (M - rB^*B)(p^{k+1} - p^{k}), p^{k+1}-p^{k} \rangle + \langle  (M - rB^*B)(p^{k+1} - p^{k}), p-p^{k-1} \rangle.   \label{eq:gap:full:p}
\end{align}
Since $(N-rA^*A)/r$ and $(M-rB^*B)/r$ are both self-adjoint and positive semi-definite operator, there exist square roots operator $S \in L(X,Z)$ and $H  \in L(Y,Z)$, which are  positive semi-definite, self-adjoint, bounded and linear (see Theorem VI.9 of \cite{SIMON}),  such that $N-rA^*A = rS^*S$ and $M-rB^*B = rH^*H$. Thus, by polarization identity, for the last term in \eqref{eq:gap:full:p}, we have
\begin{align}
&\langle  (M - rB^*B)(p^{k+1} - p^{k}), p-p^{k-1} \rangle = \langle rH^*H(p^{k+1} - p^{k}), p - p^k + p^k -p^{k-1} \rangle \notag\\
& \leq \frac{r}{2}[\|H(p^{k+1}-p^{k})\|^2 + \|H(p^k-p^{k-1})\|^2] -  \langle rH^*H(p^{k+1} - p^{k}), p^{k} -p \rangle \notag\\
& \leq  \frac{r}{2}[\|H(p^{k+1}-p^{k})\|^2 + \|H(p^k-p^{k-1})\|^2]\notag\\
&-\frac{r}{2}[\|H(p^{k+1}-p)\|^2 - \|H(p^k-p)\|^2 - \|H(p^{k+1}-p^k)\|] \notag\\
& = r\|H(p^{k+1}-p^{k})\|^2  +\frac{r}{2}\|H(p^k-p^{k-1})\|^2 -\frac{r}{2}[\|H(p^{k+1}-p)\|^2 - \|H(p^k-p)\|^2].\label{eq:full:gap:p:final}
\end{align}

With \eqref{eq:gap:full:p} and \eqref{eq:full:gap:p:final}, we have
\begin{equation}\label{eq:reorganize:inner:p}
\begin{aligned}
&-\frac{2-\rho_{k}}{\rho_{k}}\langle  p^{k} - p^{k+1}, -B^*(\bar \lambda^{k} -\bar \lambda^{k+1})\rangle + \langle (M-rB^*B)(p^k-p^{k+1}), p^{k+1}-p \rangle \\
&=\frac{2-\rho_{k}}{\rho_{k}} [-\langle  p^{k} - p^{k+1}, -B^*(\bar \lambda^{k} -\bar \lambda^{k+1})\rangle + \langle (M-rB^*B)(p^k-p^{k+1}), p^{k+1}-p \rangle]\\
&+(2-\frac{2}{\rho_k})\langle (M-rB^*B)(p^k-p^{k+1}), p^{k+1}-p \rangle \\
& \leq \frac{2-\rho_{k}}{\rho_{k}}r\{-\|H(p^{k+1}-p^{k})\|^2  +\frac{1}{2}\|H(p^k-p^{k-1})\|^2 +\frac{1}{2}[\|H(p^{k}-p)\|^2 - \|H(p^{k+1}-p)\|^2]\}, \\
&+ (2-\frac{2}{\rho_k})\frac{r}{2}(\|H(p^k-p)\|^2-\|H(p^{k+1}-p)\|^2-\|H(p^{k}-p^{k+1})\|^2) \\
& =  -\frac{r}{\rho_k}\|H(p^{k+1}-p^{k})\|^2
        +(\frac{r}{\rho_k} -\frac{r}{2})\|H(p^k-p^{k-1})\|^2 +\frac{r}{2}(\|H(p^{k}-p)\|^2 - \|H(p^{k+1}-p)\|^2).
\end{aligned}
\end{equation}
Substituting \eqref{eq:reorganize:inner:p} into \eqref{eq:gap:estimate:relaxed:rho1:origin:fullyE}, we get Corollary \ref{cor:fully:relax:F} for $\rho_{k} \in (0,1)$ case.

For $\rho_k \in [1,2)$, with Theorem \ref{them:gap:estimate:full:pre:lao}, the expansion \eqref{eq:expan:w:relax}, \eqref{eq:reorganize:inner:p} and with similar arguments as in Theorem \ref{thm:gap:estimate:pre:u:weak}, one gets this corollary.
\end{proof}

Still similar to Lemma \ref{lem:admm-fixed-points:rela}, we also have the following lemma.
\begin{lemma}
  \label{lem:fully:padmm-fixed-points:relax}
  Iteration~\eqref{admm:relax:sys:full:final} possesses the following
  properties:
    If $\wlim_{i \to \infty}(u^{k_i}, p^{k_i}, \bar \lambda^{k_i}) = (u^*, p^*,\bar \lambda^*)$ and
    $\lim_{i \to \infty} ( (N-rA^*A)(u^{k_i}-u^{k_i+1}), (M-rB^*B)(p^{k_i} - p^{k_i+1}) , B(p^{k_i} -  p^{k_i+1}), \bar \lambda^{k_i} - \bar \lambda^{k_i+1})
    = (0,0, 0,0)$, then
    $(u^{k_i+1}, p^{k_i+1}, \bar \lambda^{k_i+1})$ converges weakly to
    a saddle point for~\eqref{eq:saddle-point-prob}.
\end{lemma}

For the weak convergence of iteration \eqref{admm:relax:sys:full:final} and the ergodic convergence rate, we have similar theorem as Theorem \ref{thm:gap:estimate:pre:u:weak} and \ref{thm:douglas-rachford-ergodic-rate-relax} as follows, whose proof is omitted.

\begin{theorem}
  \label{thm:dr-weak-convergence:full:pre:relax}
  With assumptions \eqref{eq:assumption:admm} and the assumptions on $\rho_k$ for \eqref{admm:relax:sys:full:final}, if~\eqref{saddle:primal:original} possesses a solution, then for
  ~\eqref{admm:relax:sys:full:final}, the iteration sequences $(u^k, p^{k}, \bar \lambda^{k})$
  converges weakly to
  a saddle-point $x^* = (u^*,p^*, \bar \lambda^*)$ of~\eqref{eq:saddle-point-prob} and $(u^*, p^*)$ is a solution of \eqref{saddle:primal:original}. Then
  the ergodic sequences
  \begin{equation}
    \label{eq:douglas-rachford-erg-seq-fully-pre-realx}
    u_{\erg}^k = \frac1k \sum_{k' = 2}^{k+1} u^{k'},\ \ \  p_{\erg}^k = \frac1k \sum_{k' = 2}^{k} p^{k+1}, \quad
    \bar \lambda_{\erg}^k = \frac1k \sum_{k' = 2}^{k+1} \bar \lambda^{k'},
  \end{equation}
  converge weakly to $u^*$, $p^*$ and $\bar  \lambda^*$ respectively. We have $\gap_{x^*}(x_{\erg}^k) \geq 0$ with $x_{\erg}^k =  (u_{\erg}^k, p_{\erg}^k, \bar \lambda_{\erg}^k)$, and the ergodic convergence rate of the ``partial" primal-dual gap is as follows, for $0< \rho_k < 1$,
  \begin{equation}\notag
    \begin{aligned}
   \gap_{x^*}(x_{\erg}^k) &\leq \frac 1k
    \
    \Bigl[
    \frac{r \norm{B p^1 - Bp^* }^2}{2} +
    \frac{\norm{\bar \lambda^1 - \bar \lambda^*}^2}{2r} +  \frac{1-\rho_{1}}{2\rho_{1}r}\|v^{1}-v^*\|^2 +\frac{r}{2}\|S(u^{1}-u^*)\|^2 \\
   & \quad  + r(\frac{1}{\rho_1} - \frac12)\|H(p^1-p^0)\|^2+ \frac{r}{2}\|H(p^1-p^*)\|^2
    \Bigr]
    = \mO(\frac{1}{k}).
    \end{aligned}
   \end{equation}
   For $1 \leq \rho_k <2$, we have
     \begin{equation}\notag
    \begin{aligned}
  \gap_{x^*}(x_{\erg}^k) &\leq \frac 1k
     \
    \Bigl[  \frac{1}{\rho_{1}}
    \Bigl(\frac{r \norm{B p^1 - Bp^* }^2}{2} +
    \frac{\norm{\bar \lambda^1 - \bar \lambda^*}^2}{2r} +   \frac{\rho_1-1}{2}\|\sqrt{r}B(p^1-p) + (\bar \lambda^1 - \bar \lambda)/\sqrt{r} \|^2\Bigr) \\
   &\quad +\frac{r}{2}\|S(u^{1}-u^*)\|^2  + r(\frac{1}{\rho_1} - \frac12)\|H(p^1-p^0)\|^2+ \frac{r}{2}\|H(p^1-p^*)\|^2
    \Bigr]
    = \mO(\frac{1}{k}).
    \end{aligned}
  \end{equation}
\end{theorem}
Here for simplicity and clarity, we use different ergodic sequence $x_{\erg}^{k}$ \eqref{eq:douglas-rachford-erg-seq-fully-pre-realx} compared to ergodic sequences in previous sections, and $(u^1,p^1, \bar \lambda^1)$ is used in above theorem followed by the estimates involving $p^{0}$ with Corollary \eqref{cor:fully:relax:F}.

\section{Numerical Part}
\noindent\textbf{TV denosing.}
Let us apply the discrete framework for the total-variation regularized
$L^2$- and $L^1$-type denoising problems (the
$L^2$-case is usually called the \emph{ROF model})
\begin{equation}
  \label{eq:tv-denoising-min-prob}
  \min_{u \in X} \ \frac1q \| u - f \|_q^q
  + \alpha \|\nabla u\|_1, \qquad
  \qquad
  \text{for $q=1,2$}
\end{equation}
where $f: \Omega \to \RR$ is a given noisy image and $\alpha > 0$ is a
regularization parameter. For detailed discrete setting, see \cite{BS2}.

\noindent\textbf{The case $q=2$.}
We see that for $q=2$,~\eqref{eq:tv-denoising-min-prob} could be reformulated
as \eqref{saddle:primal:original} with introducing adjacent variable $p = \nabla u$ and the
data
\begin{equation}\label{eq:rof:data}
A = \nabla, \  B=-I, c = 0, \  F = \frac12 \| u - f \|_2^2, \
G = \alpha \|\cdot\|_{1}.
\end{equation}

Furthermore, observe that $(\partial F + rA^*A)u := (I - r \Delta)u - f$ is monotone for any
$r > 0$. Hence, a preconditioner for $T = I - r \Delta$ has to be chosen here.
In order to implement the algorithm,
noting that for $(rI + \partial G)^{-1}$, we have
\begin{equation}\label{resolve:G}
(rI + \partial G)^{-1}
= \Bigl(I + \frac{1}{r}\partial G \Bigr)^{-1}\Bigl(\frac1r \,\cdot\, \Bigr),
\end{equation}
and since $G = \alpha \|\cdot\|_{1}$, the latter resolvent is given by
the soft-shrinkage operator $\mathcal{S}_{\alpha/r}(\,\cdot\,/r)$,
see \cite{HBPL, BS1,BS2}.
 This can in turn be expressed, for $p \in Y$, by
\begin{equation}
  \label{eq:resolvent-infty-constraints:vector}
  \mathcal{S}_{\alpha/r}(p/r) = \frac{\max(0, |p| - \alpha)}{r |p|} p
\end{equation}
where $|p| = \sqrt{(p_1)^2 + (p_2)^2}$.
Thus, preconditioned and relaxed ADMM algorithms can be written as in Table \ref{tab:precond-l2-tv}.
\begin{table}
  \centering
  \begin{tabular}{p{0.15\linewidth}r@{\ }p{0.65\linewidth}}
    \toprule
    \multicolumn{3}{l}{\textbf{rpADMM}\ \ \textbf{Objective:} $L^2$-TV denoising}
    \hfill%
    $\min_{u \in X}
    \ \frac12 \|u - f\|_2^2 + \alpha \| \nabla u \|_1$
    \hfill\mbox{}
    \\
    \midrule

    Initialization: &
    \multicolumn{2}{l}{%
      $(u^0, p^0, \bar \lambda^0) \in X \times Y \times Y$ initial guess,
      $r > 0$ step-size, $\rho_k \equiv 1.9$,}
    \\
    &
    \multicolumn{2}{l}{%
      $n \geq 1$ inner iterations for symmetric Gauss--Seidel}
    \\[\medskipamount]
    Iteration: &
    $u^{k+1}$ & $= \SRBGS_{1, r}^n\bigl(u^k, \Div (\bar \lambda^k - r p^{k}) + f
    \bigr)$ 
    \\[\smallskipamount]
    & $p^{k+1}$ & $=  \mathcal{S}_{\alpha/r}\bigl(
    (\bar \lambda^{k} + r \rho_{k}\nabla u^{k+1} + (1-\rho_{k}) p^{k})/r \bigr)$
    \\[\smallskipamount]
    & $\bar \lambda^{k+1}$ &$ = \bar \lambda^{k} + r(\rho_{k}\nabla u^{k+1} - p^{k+1} + (1-\rho_{k}) p^{k})$
    \\[\smallskipamount]
    \bottomrule
  \end{tabular}
  \caption{The overrelaxed and preconditioned ADMM iteration for $L^2$-TV
    denoising.}
  \label{tab:precond-l2-tv}
\end{table}
We use the primal-dual gap
\begin{equation}
  \label{eq:l2-tv-gap}
  \mathfrak{G}_{L^2\text{-}\TV}(u,\bar\lambda) = \frac{\|u-f\|_2^2}{2}
  + \alpha \|\nabla u\|_1 + \frac{\|\Div \bar \lambda + f\|_2^2}{2}
  - \frac{\|f\|_2^2}{2} + \mathcal{I}_{\{\|\bar \lambda\|_\infty \leq \alpha\}}(\bar \lambda)
\end{equation}
as the basis for comparison as well as a stopping criterion by
plugging the gap $  \mathfrak{G}_{L^2\text{-}\TV}(u^{k+1},\bar \lambda^{k+1})$ in each iteration step. Furthermore the boundedness of $G^*(\bar \lambda^k)=\mathcal{I}_{\{\|\bar \lambda\|_\infty \leq \alpha\}}(\bar \lambda^k)$ and the boundedness of the gap \eqref{eq:l2-tv-gap} could be guaranteed by the following remark, i.e., Remark \ref{rem:gap:legal}.
\begin{remark}\label{rem:gap:legal}
 While applying \eqref{pre:admm:sys:final} or \eqref{pre:admm:sys:referee:relax:final} to ROF model \eqref{eq:tv-denoising-min-prob}, since $G^*(\bar \lambda^k)=\mathcal{I}_{\{\|\bar \lambda\|_\infty \leq \alpha\}}(\bar \lambda^k)$ is a constraint on the Lagrangian multipliers $\bar \lambda$, then the Lagrangian multipliers $\bar \lambda_{1}^{k}$ always satisfy the constraint, i.e., $G^*(\bar \lambda_{1}^{k}) $ is finite.
This means the primal-dual gap function as in \eqref{eq:l2-tv-gap} originally suggested in \cite{CP} for primal-dual algorithms is also bounded for \eqref{pre:admm:sys:final} or \eqref{pre:admm:sys:referee:relax:final}.
\end{remark}
\begin{proof}
Actually, by the analysis in section \ref{sec:admm-equavi-dr}, we see the update of $\bar \lambda_{1}^{k}$ as in \eqref{pre:admm:sys:final} or \eqref{pre:admm:sys:referee:relax:final} can also be equivalently written as the update of $\lambda_{1}^{k+1}$ in \eqref{pre:right} or \eqref{pre:right:rela} by applying to the dual problem \eqref{eq:dual}. With the data \eqref{eq:rof:data}, it can be checked that
\[
rA_{\tilde B}(\lambda_{1},  \lambda_{2}) = r \partial \tilde{G}^*(-\tilde B^* \cdot) + r \tilde c = (r \partial G^*(\lambda_{1}), r \partial I_{\{0\}}(\lambda_{2}) ).
\]
Thus, by the update formula of $\lambda_{1}^{k+1} $ in \eqref{pre:right} or \eqref{pre:right:rela}, we have
\[
\bar \lambda_{1}^{k} = \lambda_{1}^{k+1} = J_{rA_{\tilde B}}(v_{1}^{k})=  (I + r \partial G^*)^{-1}(v_{1}^{k}) = \mP_\alpha(v_{1}^k)  = \frac{v_{1}^k}{\max(1, |v_{1}^k|/\alpha)},
\]
and what follows is that $\lambda_{1}^{k+1}$ always satisfy the constraint, i.e., $G^*(\bar \lambda_{1}^{k})$ is always bounded. By the definition of the primal-dual gap \cite{CP}
\[
F(u^k) + G(Au^k) +  F^*(-A^*\bar \lambda_{1}^{k}) + G^*(\bar \lambda_{1}^{k}),
\]
which is also \eqref{eq:l2-tv-gap} for ROF model, we see the primal dual gap is bounded for \eqref{pre:admm:sys:final} or \eqref{pre:admm:sys:referee:relax:final} iterations. Thus this gap can be used as a stopping criterion for ROF model as introduced in \cite{CP}.
\end{proof}

\noindent\textbf{The case $q=1$.}
In this case, the discrepancy $\|\,\cdot\, - f\|_1$ is not
linear-quadratic and we have to reformulate~\eqref{eq:tv-denoising-min-prob} to
\[
\min_{u,v \in X, \  w \in Y} \ \|v - f\|_1 + \alpha \| w \|_1 \qquad
\text{subject to} \qquad
\begin{bmatrix}
  I \\ \nabla
\end{bmatrix}u =
\begin{bmatrix}
  v \\ w
\end{bmatrix}.
\]
This leads to $F = 0$ as in~\eqref{saddle:primal:original}. In
total, the problem \eqref{eq:tv-denoising-min-prob} could be reformulated to \eqref{saddle:primal:original} by introducing $p = (v,w)^{T}$ and the following data
\[
A =
\begin{bmatrix}
  I \\ \nabla
\end{bmatrix},
\quad
B =-I, \quad
c=0, \quad
F = 0, \quad
G(p) = G(v,w) = \|v - f\|_1 + \alpha \| w \|_1.
\]
Again, $rA^*A = r(I - \Delta) > 0$ for each $r > 0$, so preconditioned
ADMM is applicable. The resolvent $(rI + G)^{-1}$ then decouples into
\begin{equation}
  \label{eq:l1-tv-resolvent}
  (rI + \partial G)^{-1}(v,w) = \bigl(\mathcal{S}_{1/r}(v - rf) + f,
  \mathcal{S}_{\alpha/r}(w/r) \bigr)
\end{equation}
where the soft-shrinkage operator on $X$ is
involved which also reads as~\eqref{eq:resolvent-infty-constraints:vector} with
$|\cdot|$ denoting the usual (pointwise) absolute value.
Here, a preconditioner for $T = r(I - \Delta)$ has to be chosen.
 We use the normalized primal energy $R_k$ to realize a stopping criterion,
\begin{equation}\label{relative:energy}
 R_k := (E_k - E_*)/E_*,
 \qquad E_k =  \|u^{k}-f\|_{1} + \alpha \|\nabla u^{k}\|_{1},
\end{equation}
where $E_*$ represents the primal energy that was obtained by PADMM
after roughly $5 \times 10^5$ iterations and represents the minimal
value among all tested algorithms.

\noindent \textbf{The preconditioners.} Actually, for ROF model, efficient preconditioner is needed for the operator $T = I - r\Delta$, and for $L^1\text{-}\TV$ denoising model, efficient preconditioner is needed for the operator $T = r(I- \Delta)$. Both cases could be seen as the following general perturbed Laplace equation,
\begin{equation}\label{neumann:boudary}
  \left\{
    \begin{aligned}
      s u - r \Delta u &= b,\\
      \frac{\partial u}{\partial \nu}|_{\partial \Omega} &=  0.
    \end{aligned}
  \right.
\end{equation}%
The efficient symmetric Red-Black Gauss--Seidel method (SRBGS) is used without solving \eqref{neumann:boudary} exactly or approximately. For more details, see \cite{BS2}.

The following four classes of algorithms are tested for the ROF model, and two symmetric Red-Black Gauss--Seidel iterations are employed as a preconditioner for pDRQ and pADMM/rpADMM
  for the corresponding subproblems \eqref{neumann:boudary},
\begin{itemize}

\item ALG2: $\mathcal{O}(1/k^2)$ accelerated primal-dual algorithm
  introduced in \cite{CP} with adaptive step sizes, $\tau_{0} = 1/L$,
  $\tau_{k} \sigma_{k} L^2 = 1$, $\gamma = 0.35$ and $L = 8$.
  Here we use the same notations for the parameters $\tau$, $\sigma$
  together with $\theta$ as in \cite{CP} throughout this section.

\item pDRQ: Preconditioned Douglas-Rachford method for pure
  quadratic-linear primal (or dual) functionals in \cite{BS} with step size $\sigma = 9$.

\item ADMM/rADMM: Alternating direction method of multipliers~\eqref{admm:original:sys:final} and its overrelaxed version \eqref{admm:relax:sys:final} both with step size $r=9$, and relaxation parameter $\rho_{k} \equiv  1.9$ for \eqref{admm:relax:sys:final}.
  In each iteration step, the linear subproblem~\eqref{neumann:boudary}
  is solved by diagonalization techniques via the \emph{discrete cosine
    transform} (DCT).

\item pADMM/rpADMM: Preconditioned alternating direction method of
  multipliers according to Table~\ref{tab:precond-l2-tv} with step size $r = 9$ and relaxation parameter $\rho_{k} \equiv 1.9$.


\end{itemize}
For the $L^1$-TV denoising
problem~\eqref{eq:tv-denoising-min-prob}, i.e., $q=1$, four classes of algorithms are also tested.
\begin{itemize}

\item ALG1: $\mathcal{O}(1/k)$ primal-dual algorithm introduced in
  \cite{CP} with constant step sizes, the dual step size $\tau = 0.02$,
  $\tau\sigma L^2 = 1$, $\theta = 1$ with $L = \sqrt{8}$
  as proposed in \cite{CP}.

\item ADMM/rADMM: Alternating direction method of multipliers~\eqref{admm:original:sys:final} or its overrelaxed version as in \eqref{admm:relax:sys:final}, both with step size $r = 20$, and the relaxation parameter $\rho_{k} \equiv  1.9$ for \eqref{admm:relax:sys:final}.
  Again, DCT is employed for linear subproblem \eqref{neumann:boudary} instead of preconditioners.

\item pADMM/rpADMM: Preconditioned alternating direction method of
  multipliers as in Table \ref{tab:precond-l1-tv} with or without overrelaxation, both with step size $r = 20$, and the relaxation parameter $\rho_{k} \equiv  1.9$ for rpADMM.  Two symmetric Red-Black Gauss--Seidel iterations are
  employed as preconditioners for the linear subproblem \eqref{neumann:boudary}.

\item fADMM/fpADMM: Here fADMM represents the overrelaxed ADMM of the type in \cite{FG, DY, LST} with relaxation parameter $\tau = 1.618$, which was developed by Fortin and Glowinski \cite{FG}. DCT is employed for the linear subproblem \eqref{neumann:boudary}, and the step size $r = 20$. fpADMM represents the overrelaxed ADMM of the same type in \cite{FG}, \cite{DY, LST} with the same step size and relaxation parameter $\tau$, along with preconditioning. Again,
    two symmetric Red-Black Gauss--Seidel iterations are employed as preconditioner for the
  corresponding subproblems \eqref{neumann:boudary}. In detail, for fpADMM, one just need to set all $\rho_{k}$ as $1.0$, and replace the update of $\bar\lambda^{k+1}$ in Table \ref{tab:precond-l1-tv} as follows
  \[
\bar\lambda_{v}^{k+1} = \bar \lambda_{v}^{k} + \tau r( u^{k+1} - v^{k+1}), \quad
   \bar \lambda_{w}^{k+1} = \bar \lambda_{w}^{k} + \tau r(\nabla u^{k+1}  - w^{k+1}).
  \]

\end{itemize}
\begin{table}
  \centering
  \begin{tabular}{p{0.15\linewidth}r@{\ }p{0.65\linewidth}}
    \toprule
    \multicolumn{3}{l}{\textbf{rpADMM}\ \ \textbf{Objective:} $L^1$-TV regularization}
    \hfill%
    $\min_{u \in X}
    \  \|u - f\|_{1} + \alpha \| \nabla u \|_1$
    \hfill\mbox{}
    \\
    \midrule

    Initialization: &
    \multicolumn{2}{l}{%
      $(u^0, v^0, w^{0}, \bar \lambda_{v}^0, \bar \lambda_{w}^{0}) \in X
      \times (X \times Y)^2$ initial guess, $r > 0$ step-size,}
    \\
    &
    \multicolumn{2}{l}{%
      $n \geq 1$ inner iterations for symmetric Gauss--Seidel, $\rho_k \equiv 1.9$,}
    \\[\medskipamount]
    Iteration: &
    $b^{k}$ & $= rv^k - \bar \lambda_v^k + \Div (\bar \lambda^k_w - r w^{k})$
    \\[\smallskipamount]
    & $u^{k+1}$ & $= \SRBGS_{r,r}^n(u^k, b^k)$
    \\[\smallskipamount]
    & $v^{k+1}$ & $= \mathcal{S}_{1/r}\bigl((\bar \lambda_{v}^{k} + r(\rho_{k}u^{k+1} +(1-\rho_k) v^{k}-f))/r
    \bigr) + f$
    \\[\smallskipamount]
    & $w^{k+1}$ &$ = \mathcal{S}_{\alpha/r}\bigl( (\bar \lambda_{w}^{k}
    + r (\rho_k \nabla u^{k+1} +(1-\rho_k)w^{k} )/r\bigr)$
    \\[\smallskipamount]
   & $\bar \lambda_{v}^{k+1}$ &$=\bar \lambda_{v}^{k} + r(\rho_k u^{k+1} +(1-\rho_k) v^k- v^{k+1})$
    \\[\smallskipamount]
   & $\bar \lambda_{w}^{k+1}$ &$= \bar \lambda_{w}^{k} + r(\rho_k \nabla u^{k+1} +(1-\rho_k)w^k - w^{k+1})$
    \\[\smallskipamount]
    \bottomrule
  \end{tabular}
  \caption{The overrelaxed and preconditioned ADMM iteration for $L^1$-TV
    denoising.}
  \label{tab:precond-l1-tv}
\end{table}

\begin{table}
\centering 
\begin{tabular}{lr@{\,}r@{\,}lr@{\,}r@{\,}lr@{\,}r@{\,}lr@{\,}r@{\,}l} 
\toprule
& \multicolumn{6}{c}{$\alpha = 0.1$}
& \multicolumn{6}{c}{$\alpha = 0.3$}\\
\cmidrule{2-7} 
\cmidrule{8-13}
& \multicolumn{3}{c}{$\varepsilon = 10^{-5}$}
& \multicolumn{3}{c}{$\varepsilon = 10^{-7}$}
& \multicolumn{3}{c}{$\varepsilon = 10^{-5}$}
& \multicolumn{3}{c}{$\varepsilon = 10^{-7}$}\\
\cmidrule{1-13} 
ALG2 && 46 &(1.05s) && 214&(4.03s)&& 194 &(3.49s)&& 845 &(17.32s)\\
pDRQ && 41 & (0.83s) &&  135 & (2.68s) && 73 & (1.53s) &&  941 & (18.91s)\\
ADMM  && 40&(3.31s) &&128 & (10.29s) && 68 & (5.20s) && 915& (70.00s) \\
rADMM  && 23 & (1.93s) && 69 & (5.62s) && 39 & (3.12s) && 482 & (37.90s) \\
pADMM  && 41&(1.26s) && 134 & (4.65s) && 72 & (2.39s) &&919 &(29.09s) \\
rpADMM  && 25 & (0.93s) && 76 & (2.90s) && 48 & (1.57s) && 508 & (18.92s) \\
\bottomrule 
\end{tabular}
\caption{
Numerical results for the $L^2$-TV image denoising (ROF) problem \eqref{eq:tv-denoising-min-prob} with noise level $0.1$ and regularization parameters $\alpha = 0.1$, $\alpha = 0.3$.  For all algorithms, we use the two pairs $\{k(t)\}$ to represent the iteration number $k$, CPU time $t$. The iteration is performed until the normalized primal-dual gap~\eqref{eq:l2-tv-gap} is below $\varepsilon$.} 
\label{tab:peninsula:0.5}
\end{table}
\begin{figure}
\begin{center}
\subfloat[Original image]
{\includegraphics[width=3.5cm]{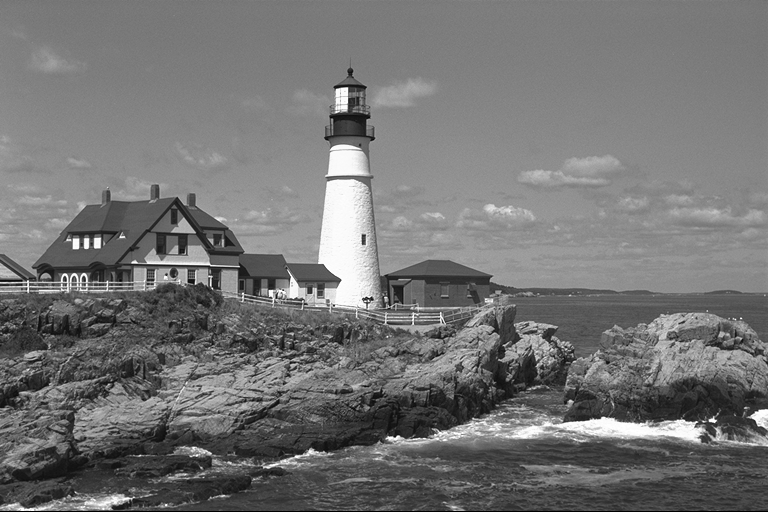}}\
\subfloat[Noisy image ($10\%$)]
{\includegraphics[width=3.5cm]{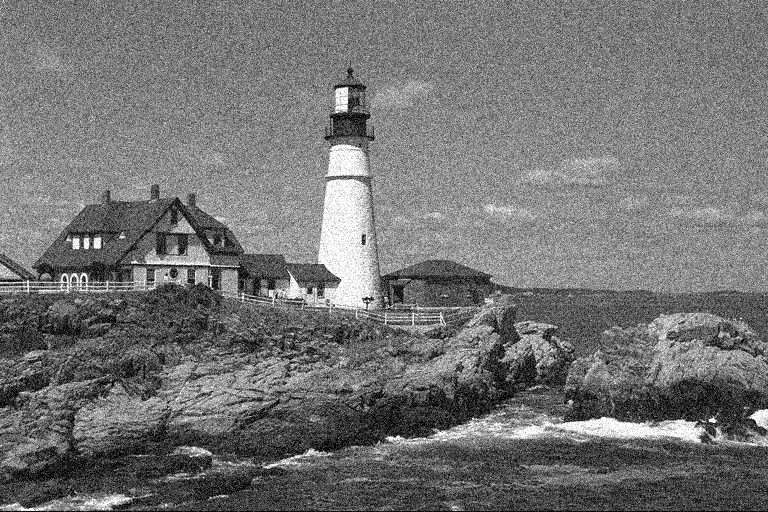}}\
\subfloat[rADMM($10^{-7}$)]
{\includegraphics[width=3.5cm]{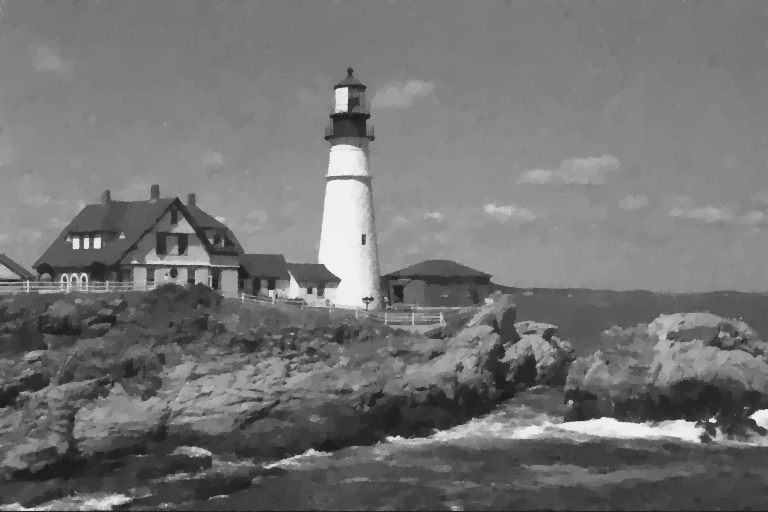}} \
\subfloat[rpADMM($10^{-7}$)]
{\includegraphics[width=3.5cm]{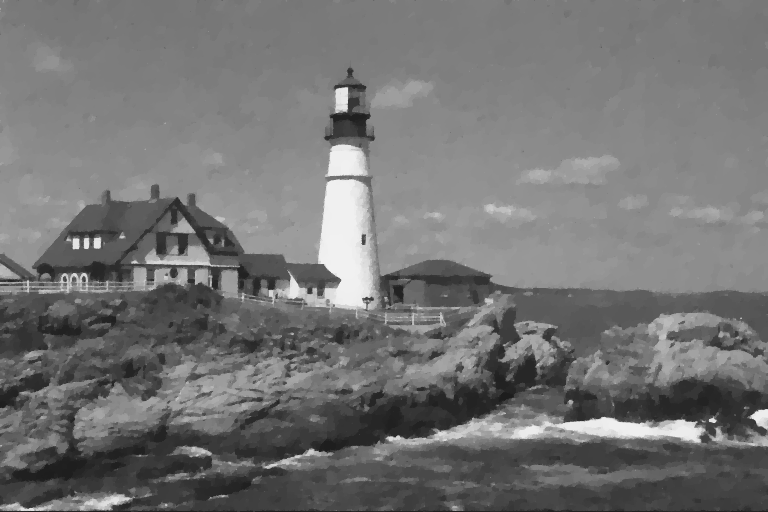}}
\end{center}
\caption{Image denoising with random Gaussian noise. (a) shows the $768\times 512$ input Peninsula image and (b) is a noisy image which has been corrupted by 10$\%$ random Gaussian noise. (c) and (d) are the denoised images by the rADMM and rpADMM algorithms with normalized primal-dual gap less than $10^{-7}$ for
  $\alpha = 0 .1$.}
\label{peninsula:l2}
\end{figure}

\begin{figure}[!htb]
  \begin{center}
    \subfloat[Numerical convergence rate of normalized primal-dual gap compared with iteration number.]
    {\includegraphics[width=7.2cm,page=1]{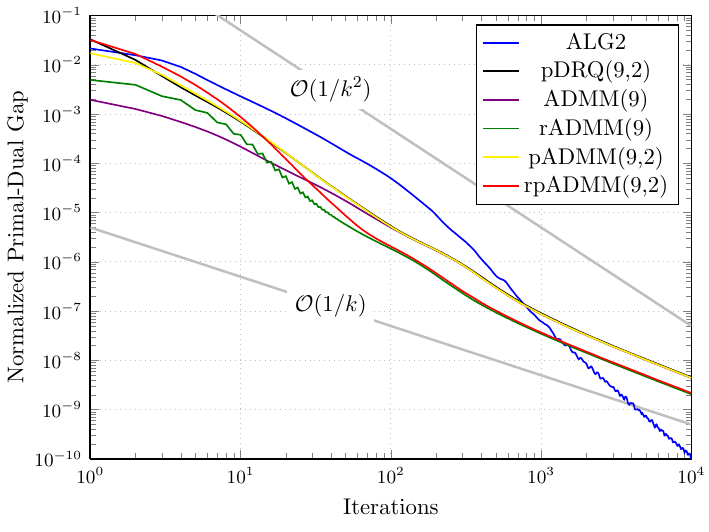}}\hfill
    \subfloat[Numerical convergence rate of normalized primal-dual gap
    compared with iteration time.] 
    {\includegraphics[width=7.2cm,page=2]{figures}}
  \end{center}
  \caption{$L^2$-$\TV$ denoising: Numerical convergence rates.
    The normalized primal-dual gap is compared in terms of iteration
    number and computation time for 
    Figure~\ref{peninsula:l2}
    with  $\alpha = 0.3$.
    The notations ADMM($r$), rADMM($r$), pADMM($r,n$), rpADMM($r,n$) and pDRQ($r,n$) are used to
    indicate the step-size $r$ and $n$ inner iterations.
    Note the double-logarithmic and
    semi-logarithmic scale, respectively, which is used in the plots.}
    \label{peninsula:denoise:l2rate}
\end{figure}

\begin{table}
  \centering
  \begin{tabular}{l@{\,}lr@{\,}lr@{\,}lr@{\,}l} 
    \toprule
    \multicolumn{8}{c}{$\alpha = 1.0$} \\
    \cmidrule{3-8} 
    & & \multicolumn{2}{c}{$\varepsilon = 10^{-4}$}
    & \multicolumn{2}{c}{$\varepsilon = 10^{-5}$}
    & \multicolumn{2}{c}{$\varepsilon = 10^{-6}$}\\
    \cmidrule{1-8} 
    ALG1 & & 341&(4.67s) & 871&(11.55s)  & 3446&(34.13s)\\ 
     ADMM &  & 143& (9.82s) & 351&(22.66s) & 1371&(82.98s)\\
     rADMM &  & 104& (7.26s) & 199&(13.65s) & 725&(53.20s)\\
     fADMM &  & 115& (8.18s) & 265&(18.48s) & 869&(62.43s)\\
     pADMM &  & 168& (3.76s) & 397&(8.74s) & 1420&(27.80s)\\
    fpADMM &  & 140& (2.16s) & 339&(5.12s) & 938&(18.43s)\\
    rpADMM &  & 107& (1.97s) & 241&(4.06s)  & 776&(15.66s)\\
    \bottomrule 
  \end{tabular}
  \caption{$L^1$-$\TV$ denoising: Performance comparison. The results are
    shown in the format $\{k(t)\}$ where $k$ and $t$ denote iteration number
    and CPU time, respectively. The iteration is performed until the normalized primal energy $R_k$ is below $\varepsilon$.}
  \label{tab:shooter-l1denoise:0.6}
\end{table}

\begin{figure}[!htb]
\begin{center}
\subfloat[Original image]
{\includegraphics[width=3.5cm]{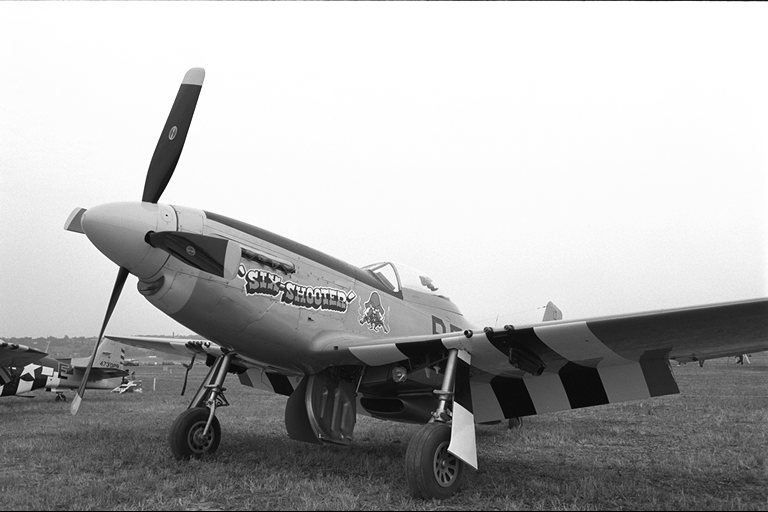}} \
\subfloat[Noisy image ($25\%$)]
{\includegraphics[width=3.5cm]{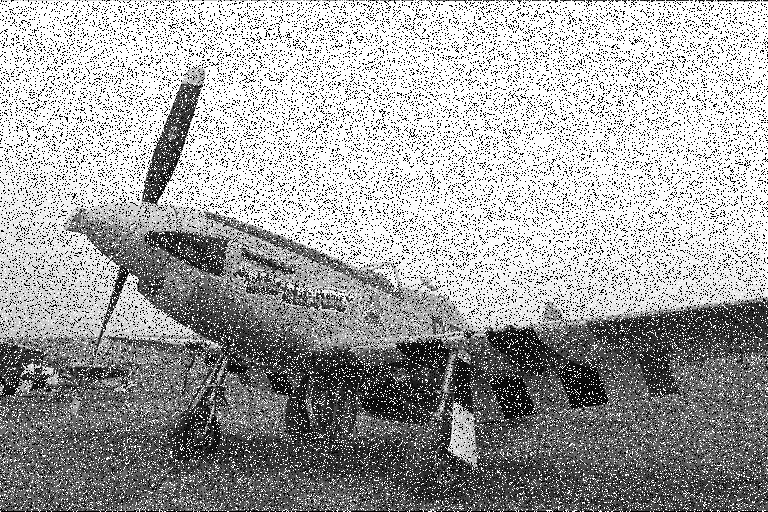}}\
\subfloat[rADMM($10^{-6}$)]
{\includegraphics[width=3.5cm]{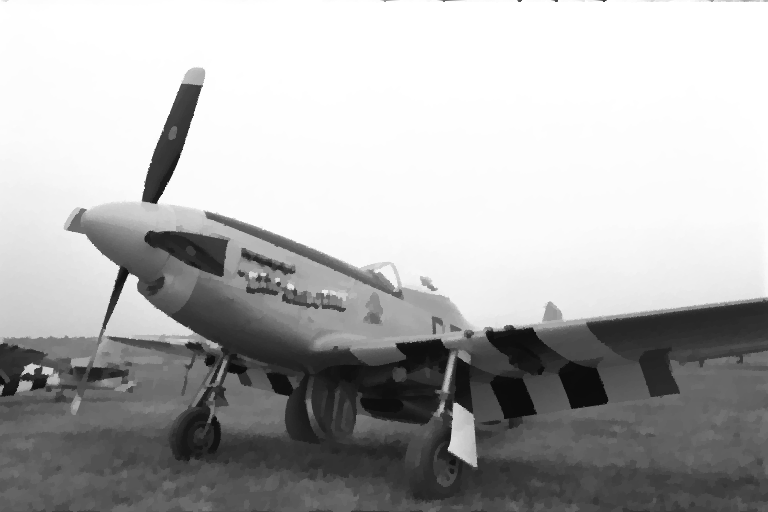}} \
\subfloat[rpADMM($10^{-6}$)]
{\includegraphics[width=3.5cm]{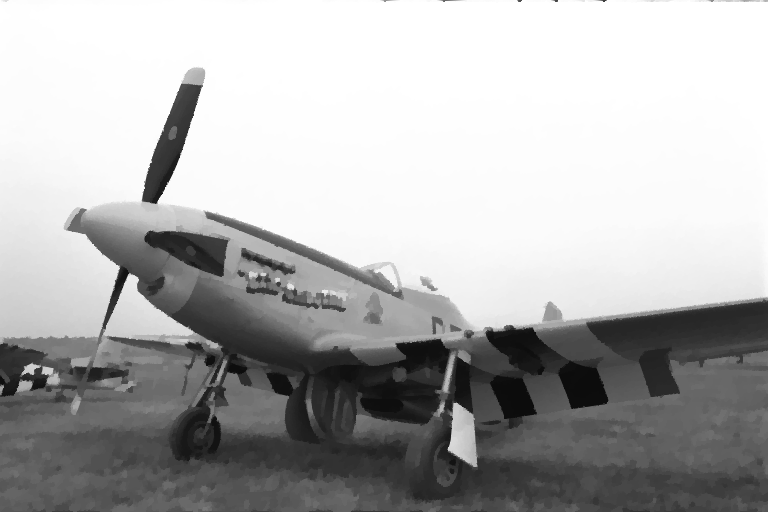}}
\end{center}
\caption{Image denoising with random salt and pepper noise. (a) shows the $768\times 512$ input Shooter image and (b) is a noisy image which has been corrupted by 25$\%$ salt and pepper noise. (c) and (d) are the denoised images by the rADMM and rpADMM algorithms with normalized primal energy less than $10^{-6}$ for $\alpha = 1.0$.}
\label{shooter:l1}
\end{figure}

\begin{figure}[!htb]
  \begin{center}
    \subfloat[Numerical convergence rate of normalized error of primal energy compared with iteration number.]
    {\includegraphics[width=7.2cm,page=3]{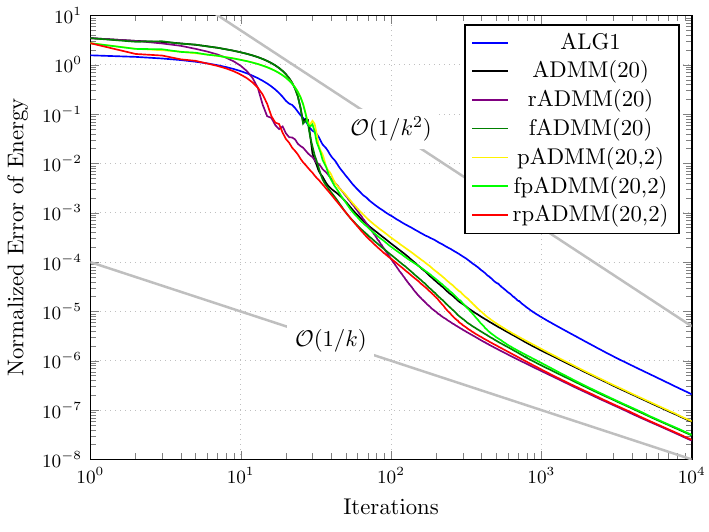}}\hfill
    \subfloat[Numerical convergence rate of normalized error of primal energy
    compared with iteration time.] 
    {\includegraphics[width=7.2cm,page=4]{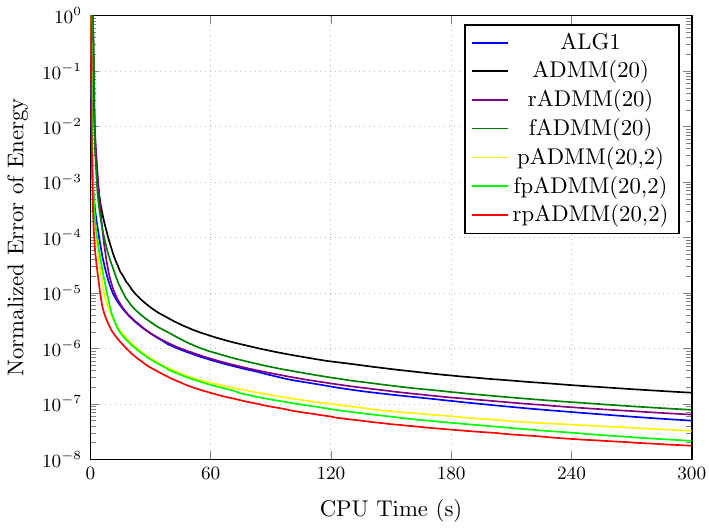}}
  \end{center}
  \caption{$L^1$-$\TV$ denoising: Numerical convergence rates.
    The normalized error of primal energy is compared in terms of iteration
    number and computation time for 
    Figure~\ref{shooter:l1}
    with  $\alpha = 1.0$.
    The notations 
    ADMM($r$), rADMM($r$), pADMM($r,n$), rpADMM($r,n$), fADMM($r$), and fpADMM($r,n$) are used to
    indicate the step-size $r$ and $n$ inner iterations.
    Note the double-logarithmic and
    semi-logarithmic scale, respectively, which is used in the plots.}
   \label{lena:denoise:l1rate}
\end{figure}

\noindent\textbf{TGV denosing.} The TGV regularized $L^2$ deniosing model reads as follows \cite{BPK},
\begin{equation}
  \label{eq:tgv_primal}
  \min_{u \in X} \ \frac{\norm[2]{x - f}^2}{2} + \TGV_{\alpha}^{2}(x),
\end{equation}
where the TGV regularization could be written as follows,
\begin{equation}\label{tgv:bv}
  \TGV_{\alpha}^{2}(x) = \min_{w \in \BD(\Omega)} \ \alpha_{1} \| \nabla x -w\|_{\mathcal{M}} + \alpha_{0} \|\mathcal{E}w\|_{\mathcal{M}},
\end{equation}
where $\BD(\Omega)$ denotes the space of vector fields of
{\it Bounded Deformation}, and the weak symmetrized derivative $\mathcal{E}w
:= (\nabla w + \nabla w^T)/2$ being a matrix-valued Radon measure.
Moreover, $\|\cdot\|_{\mathcal{M}}$ denotes the Radon norm for the
corresponding vector-valued and matrix-valued Radon measures.
Denoting $u = (x,w)^{T}$, $p = (v,q)^{T}$, as in \eqref{saddle:primal:original}, we have the data
\begin{equation}
  \label{eq:tgv-denoising-pdr-data}
  A =
  \begin{bmatrix}
    \grad & -I \\
    0 & \symgrad
  \end{bmatrix},
  \
  B =   \begin{bmatrix}
    -I & 0 \\
    0 & -I
  \end{bmatrix}, \
  \ F(u) = F(x), \ \
  G(p) = \alpha_{1} \|v\|_{\mathcal{M}}
  + \alpha_0 \|q\|_{\mathcal{M}}, \  c=  \begin{bmatrix}
    0 \\
    0
  \end{bmatrix}.
\end{equation}

The following four classes algorithms are used for $L^2$-TGV tests.
We use one time symmetric Red-Black block Gauss-Seidel iteration \cite{BS2} for the corresponding linear subproblems with $(x,w)^{T}$. We choose the step size $r =3$ uniformly for pADMM, fpADMM, rpADMM and the relaxation parameter settings of fpADMM and rpADMM are the same as in TV case.
For more details for implementations and the preconditioners, see \cite{BS2}. The stopping criterion is the normalized primal energy of \eqref{eq:tgv_primal} where minimal energy was obtained
after roughly $2 \times 10^4$ iterations and represents the minimal
value among all tested algorithms.
The parameter settings of ALG1 are  as follows:
\begin{itemize}
\item ALG1: $\mathcal{O}(1/k)$ primal-dual algorithm
  introduced in \cite{CP1} with constant step sizes, dual step size $\tau = 0.05$,
  $\tau \sigma L^2 = 1$, $L = \sqrt{12}$.
\end{itemize}

\begin{figure}[!htb]
\begin{center}
\subfloat[Original image]
{\includegraphics[width=3.5cm]{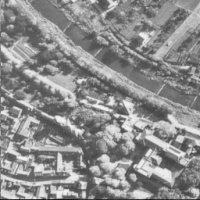}} \
\subfloat[Noisy image ($5\%$)]
{\includegraphics[width=3.5cm]{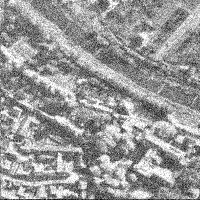}}\
\subfloat[rpADMM($10^{-3}$)]
{\includegraphics[width=3.5cm]{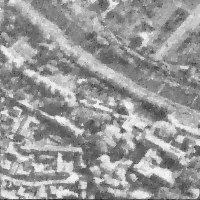}} \
\subfloat[rpADMM($10^{-5}$)]
{\includegraphics[width=3.5cm]{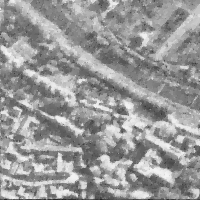}}
\end{center}
\caption{Image denoising with random salt and pepper noise. (a) shows the $200\times 200$ input Aerial image and (b) is a noisy image which has been corrupted by 5$\%$ gaussian noise. (c) and (d) are the denoised images by rpADMM algorithms with normalized primal energy $R_k$ less than $10^{-3}$ and $10^{-5}$ for $\alpha_0 = 0.1$, $\alpha_1 = 0.05$.}
\label{shooter:l1}
\end{figure}

\begin{table}
\centering 
\begin{tabular}{lr@{\,}r@{\,}lr@{\,}r@{\,}lr@{\,}r@{\,}lr@{\,}r@{\,}l} 
\toprule
& \multicolumn{6}{c}{$\alpha_1 = 0.05$, noise level: 5\%}
& \multicolumn{6}{c}{$\alpha_1 = 0.1$, noise level: 10\%}\\
\cmidrule{2-7} 
\cmidrule{8-13}
& \multicolumn{3}{c}{$\varepsilon = 10^{-3}$}
& \multicolumn{3}{c}{$\varepsilon = 10^{-5}$}
& \multicolumn{3}{c}{$\varepsilon = 10^{-3}$}
& \multicolumn{3}{c}{$\varepsilon = 10^{-5}$}\\
\cmidrule{1-13} 
ALG1&& 158&(0.97s)  && 3090&(21.53s) && 169&(1.24s) && 4201&(32.12s)\\
\midrule 
pADMM &&78&(0.94s) &&1820&(23.21s) &&83&(1.01s) && 2343&(30.80s)\\
fpADMM && 66&(0.82s)  && 1152&(14.62s) && 61&(0.68s) && 1470&(18.57s)\\
rpADMM && 51&(0.63s)  && 966&(13.06s) && 50&(0.60s) && 1250&(15.86s)\\
\bottomrule 
\end{tabular}

\vspace*{-0.5em}
\caption{  Numerical results for $L^2$-TGV image denoising problem
  \eqref{eq:tgv-denoising-pdr-data} with noise level 5\%, regularization parameters
  $\alpha_0 = 0.1$, $\alpha_1 = 0.05$, and noise level 10\%, regularization parameters
  $\alpha_0 = 0.2$, $\alpha_1 = 0.1$. For all algorithms, the pair
  $\{k(t)\}$ is used to represent the iteration number $k$ and time cost $t$.
  The iteration is performed until the normalized primal energy
  is below $\varepsilon$.}
\label{tab:l2tgv_denoising}
\end{table}

\begin{figure}[!htb]
  \begin{center}
    \subfloat[Numerical convergence rate of normalized error of primal energy compared with iteration number.]
    {\includegraphics[width=7.2cm,page=5]{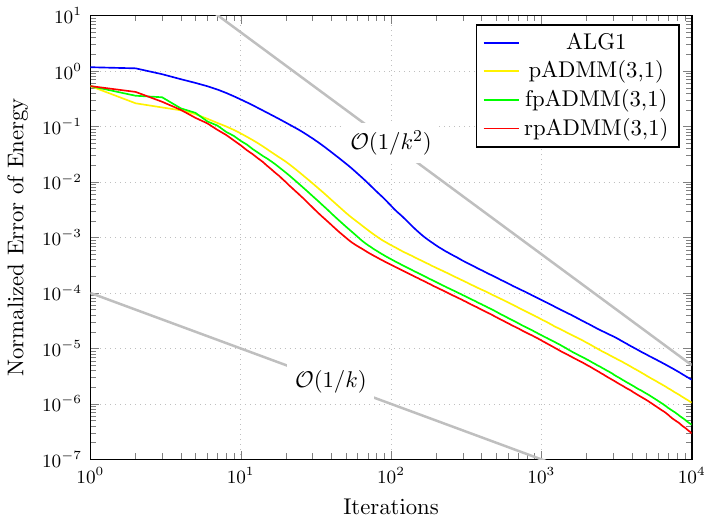}}\hfill
    \subfloat[Numerical convergence rate of normalized error of primal energy
    compared with iteration time.] 
    {\includegraphics[width=7.2cm,page=6]{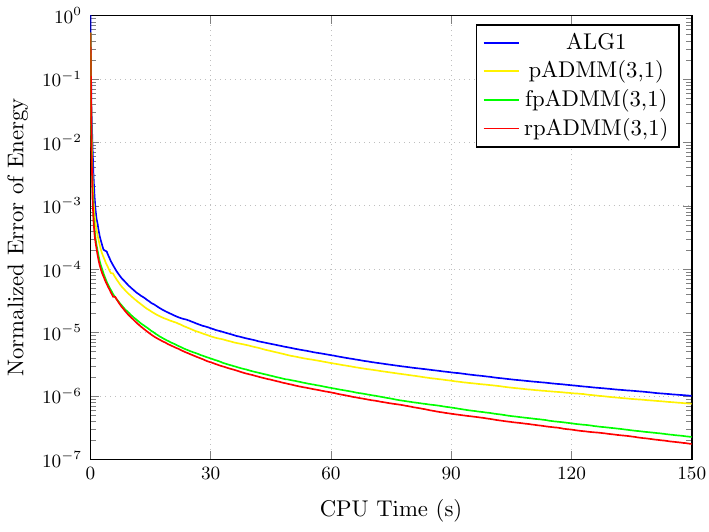}}
  \end{center}
  \caption{$L^2$-$\TGV$ denoising: Numerical convergence rates.
    The normalized error of primal energy is compared in terms of iteration
    number and computation time $\alpha_0 = 0.2$, $\alpha_1 = 0.1$.
    The notations 
    pADMM($r,n$), fpADMM($r,n$) and rpADMM($r$) are used to
    indicate the step-size $r$ and $n$ inner iterations.
    Note the double-logarithmic and
    semi-logarithmic scale, respectively, which is used in the plots.}
   \label{tgv:denoise:tim3rate}
\end{figure}

From Table \ref{tab:peninsula:0.5}, Table \ref{tab:shooter-l1denoise:0.6}, Table \ref{tab:l2tgv_denoising}, Figure \ref{peninsula:denoise:l2rate}, Figure \ref{lena:denoise:l1rate} and Figure \ref{tgv:denoise:tim3rate}, for both ROF, $L^1\text{-}\TV$ model and $L^2$-TGV model, it could be seen that overrelaxation with relaxation parameters setting $\rho_{k} \equiv 1.9$ could bring out certain acceleration, at least $30\%$ faster than $\rho_{k} \equiv 1.0$. Efficient preconditioner could bring out more efficiency, which is also observed in \cite{BS1} through a proximal point approach. Generally, the preconditioned ADMM with overrelaxation could bring out 3-5 times faster than the original ADMM without any preconditioning and relaxations. For $L^1\text{-}\TV$, it could be seen that, there is no big difference between rADMM and fADMM, or between rpADMM and fpADMM, which employ different relaxation strategies. Besides, rADMM or rpADMM could be slightly faster with fewer number of iterations in some cases.
\section{Conclusions}
We mainly studied the weak convergence and ergodic convergence rate respecting to the partial prinal-dual gap of an overrelaxed and fully preconditioned ADMM method for general
non-smooth regularized problems. Besides, we could get certain accelerations by overrelaxation strategy together with preconditioning. We also discussed relationships between a kind of relaxed ADMM with preconditioning and the corresponding Douglas-Rachford splitting methods.
Applications for image
denoising are also presented. The numerical tests showed that overrelaxed and preconditioned ADMM has
the potential to bring out appealing benefits and fast algorithms, especially where the data fitting terms are not strongly convex.
Nevertheless, there is still some interesting questions for further investigations, for example, the accelerations as proposed in \cite{XU}, the convergence properties in functional spaces, especially the Banach spaces \cite{HRH}.


\section*{Acknowledgements}
{\small
The author is very grateful to Kristian Bredies, who suggested numerous improvements during preparing the draft of this article.
He also acknowledges the support of
Fundamental Research Funds for the Central Universities, and the
research funds of Renmin University of China (15XNLF20) and NSF of China under grant No. \,11701563.
}

  \end{document}